\documentclass[a4paper]{amsart}
\usepackage[english]{babel}
\usepackage{amssymb}
\usepackage{MnSymbol}
\usepackage{amsmath}
\usepackage{amsthm}
\usepackage{amsaddr}
\usepackage{graphicx}
\usepackage{fancyhdr}
\usepackage{bbm}
\usepackage{multirow}
\usepackage{enumitem}
\usepackage{algorithm2e}
\usepackage{soul}
\newfont{\bssten}{cmssbx10}
\newfont{\bssnine}{cmssbx10 scaled 900}
\newfont{\bssdoz}{cmssbx10 scaled 1200}

\usepackage{latexsym,amssymb,amsthm,amsxtra}
\usepackage{amsmath}
\usepackage{stmaryrd}
\usepackage{mathrsfs}
\usepackage{tikz}
\usepackage{pgf}
\usepackage{cancel}
\usepackage{caption} 
\usepackage{dsfont}

\newtheorem{corollaire}{Corollary}
\newtheorem{theoreme}{Theorem}
\newtheorem{lemme}{Lemma}
\newtheorem{proposition}{Proposition}
\newtheorem{definition}{Definition}

\theoremstyle{definition}
\theoremstyle{remark}

\newtheorem{proof(theorem)}{Proof(Theorem)}
\newtheorem{remark}{Remark}
\usepackage{tabularx} 
\usepackage{multirow} 
\usepackage{multicol} 
\usepackage{arydshln} 
\usepackage{fancybox} 
\usepackage{multicol} 
\usepackage{array} 
\usepackage{fancybox}
\usepackage{listings}
\usepackage{array,multirow,makecell}
\setcellgapes{1pt}
\makegapedcells
\newcolumntype{R}[1]{>{\raggedleft\arraybackslash }b{#1}}
\newcolumntype{L}[1]{>{\raggedright\arraybackslash }b{#1}}
\newcolumntype{C}[1]{>{\centering\arraybackslash }b{#1}}

\newcommand{\Pb}{\mathbb{P}}
\newcommand{\N}{\mathbb{N}}

\newcommand{\Z}{\mathbb{Z}}
\newcommand{\R}{\mathbb{R}}

\newcommand\suite[1]{\left(#1\right)_{n\in\N}}
\newcommand\suitez[1]{\left(#1\right)_{n\in\Z}}
\newcommand\suitek[1]{\left(#1\right)_{n\ge k}}

\newcommand\pr[1]{{\mathbb P}\left[#1\right]}
\newcommand\esp[1]{{\mathbb E}\left[#1\right]}
\renewcommand\P{\mathbb P}
\newcommand{\hX}{\widehat{\mathbb{X}}}
\newcommand{\X}{\mathbb{X}}
\newcommand{\Y}{\mathbb {Y}}
\newcommand{\V}{\mathbb {V}}
\newcommand{\XX}{\tilde{\mathbb {X}}}

\title[Perfect sampling of matching models]{Perfect sampling of stochastic matching models with reneging}
\author{Thomas Masanet}
\address{IECL, Université de Lorraine / INRIA PASTA}
\author{Pascal Moyal}
\address{IECL, Université de Lorraine / INRIA PASTA}

\begin{document}
\RestyleAlgo{ruled}

\maketitle
    
\begin{abstract}
In this paper, we introduce a slight variation of the Dominated Coupling From the Past algorithm (DCFTP) of Kendall, for bounded Markov chains. It is based on the control of a (typically non-monotonic) stochastic recursion by a (typically monotonic) one. 
We show that this algorithm is particularly suitable for stochastic matching models with bounded patience, a class of models for which the steady state distribution of the system is in general unknown in closed form. We first show that the Markov chain of this model can be easily controlled by an infinite-server queue. We then investigate the particular case where patience times are deterministic, and this control argument may fail. In that case we resort to an ad-hoc technique that can also be seen as a control (this time, by the arrival sequence). We then compare this algorithm to the primitive CFTP one and to the control by an infinite server queue, and show how our perfect simulation results can be used to estimate, and compare, for instance, the loss probabilities of various systems in equilibrium. 
\end{abstract}

\section{Introduction}
\label{sec:intro}
The study of stochastic matching models is currently a very active line of research in applied probability. It has been demonstrated in various contexts, that these stochastic models are suitable to capture the dynamics of a wide range of real-time random systems, in which elements enter the system at (possibly) random times, with a view to finding a {\em match}, that is identified as such following specified compatibility rules, given by a compatibility graph between classes of items. 
Then, matched couples leave the system right away as soon as they found 
a match. This is the case in various applications, such as, peer-to-peer applications, job-search, public housing or college allocations, organ transplants, blood banks, car sharing, assemble-to-order systems, and so on.   These models have been introduced in \cite{BGM13} for bipartite graphs (which is suitable for supply/demands-type applications) and arrivals by couples, as a variant of the seminal works \cite{CKW09,AW12}. To account for a wider range of applications (e.g., dating websites, crossed kidney transplants, assemble-to-order systems or car-sharing), they have been generalized to general graphs (with simple arrivals) in \cite{mairesse2016stability}, and then to hypergraphs in \cite{NS19,RM21} and graphs with self-loops, in \cite{BMMR21}. 

Applications such as organ transplants are subject to very strong timing constraints: the patients waiting for a transplant have a finite life time in the system, and similarly, available organs are highly perishable, and must be transplanted very quickly. Hence the need to incorporate an impatience (or reneging) parameter to the system. More precisely, in this paper we address 
a general stochastic matching model, as defined in \cite{mairesse2016stability}, in which the elements have a finite (and possibly random) patience upon arrival, before the end of which they must find a match. Otherwise, they renege and leave the system forever. 
Matching models with impatience have recently been addressed for 
a bipartite model and the `N' graph in \cite{CNY20} for a matching policy of the `First Come, First Matched' ({\sc fcfm}) type, and from the point of view of stochastic optimization, for partially static policies, in \cite{ADWW21}. 
On another hand, in \cite{jonckheere2020generalized}, stability conditions, together with moment bounds at equilibrium, have been given for models in which some, but not all, classes of items are impatient, and the matching policy is of the `Max-weight' class. 

However, it is important to observe, first, that the exact computation of the stability regions of matching models is difficult for general graphs, and heavily depends on the matching policy, see e.g. \cite{mairesse2016stability,MoyPer17}. Second, the stationary distributions of the models at hand are in general unknown, and little is known about the characteristics of the steady state. In the existing literature, the models implementing the {\sc fcfm} policy constitute the only exception, in which the stationary distribution is known explicitly. It can often be characterized in a product form, as is shown using dynamic reversibility arguments (see, along the various models, \cite{ABMW18,AKRW18,BMMR21,MBM21,comte22}), for models without reneging. Let us also observe the recent advances in 
\cite{CMB21,BMM21} concerning the invariance of stationary matching rates on the matching policy for various graphs - thereby showing that all matching policies have the same matching rates as FCFM. 

However, in the cases of models with reneging, aside from the particular graph geometries addressed in \cite{CNY20}, no exact results are known. Moreover, {\sc fcfm} policies are clearly not always the best option in a real-time context: coming back to the case of organ transplants, other criteria must be taken into account, such as the level of emergency, equity, ages of the patients/donor, various levels of compatibilities, and so on. Mimicking the various existing results in queueing theory, implementing policies of the `Match the Longest queue' ({\sc ml}) or `Earliest Deadline First' ({\sc edf}) type may be profitable to minimize loss, and it is significant that {\sc edf} does {\em not} amount to {\sc fcfm} if the patience times are random. 

Our aim is to analyse matching models with reneging in steady state, for general matching policies. In view of the above discussion, we thus need to assess the stationary distribution of the matching model at hand, without knowledge of this distribution in closed form. As is well known, this task can be handled by {\em simulating perfectly} this steady state.

Perfect simulation has been a constantly active line of research in the analysis of stochastic systems, since the pioneering works of Propp and Wilson \cite{PW96}, and Borovkov and Foss \cite{BF92,BF94}. The underlying idea is now well known: Consider a discrete-event stochastic system whose stationary distribution is intractable mathematically. Then we can study the system in steady state, by {\em precisely} simulating samples of the stationary distribution, even though the latter is not known in closed form, instead of approximating it by long-run trajectories. Then, various average performance parameters at equilibrium can be 
assessed by Monte-Carlo techniques. 

The celebrated Propp and Wilson algorithm \cite{PW96} is based on coupling-from-the-past (CFTP), namely, all trajectories of the considered Markov chain coalesce before time 0, whenever these trajectories are initiated from all possible states of the chain, far away enough in the past. This phenomenon is closely related to the concept of 
strong backwards coupling (see e.g. \cite{borovkov98} and Chapter 2.5 of \cite{BacBre02}), and the connections between the two notions are investigated for various cases of stochastic recursions in \cite{FT98}. Strong backwards coupling is the pillar of the construction of the stationary state under general non-Markov assumptions, {\em via} the use of renovating events, see e.g. \cite{BF92,BF94}. It is also a tool to construct stationary states on enriched probability spaces, via skew-product constructions, see \cite{lisek1982,AK99,moyal2015}. 

As they rely on the exact coalescence of a family of Markov chains, CFTP algorithms are typically adapted to finite-state spaces and to monotonic dynamics, using envelope techniques. Various authors have extended these settings: generalizing the ideas in \cite{FT98}, it is proven in \cite{ken04} that geometrically ergodic Markov chains admit a CFTP algorithm of the envelope-type, even if they are not monotonic, a result that was then generalized to a wider class of ergodic Markov chains, in \cite{CK07}. Various related approaches have then been proposed, that all rely on the following intuitive idea: simulating from the past a more `simple' recursion, and deducing the stationary state of the recursion of interest by comparison. This is the core idea of the so-called `Dominated coupling from the past' (DCFTP) introduced in \cite{Ken98,KM00} and then \cite{ken04}, of the so-called `Bounding chains' of Huber \cite{huber2004,huber2016}, that are particularly adapted to mechanical-statistical contexts, and of various envelope techniques for queueing systems, see e.g. \cite{BGV08}. 
More recently, DCFTP-related methods has been implemented, together with saturation techniques, to perfectly simulate non-Markov queueing systems, see \cite{BD15} for infinite-server and loss queues, and \cite{BD18} for multiple-server queues. 

This paper is a first contribution to the perfect sampling of stochastic matching models. We introduce two perfect sampling algorithms, Algorithms \ref{algo2} and \ref{algo3} below, that produce samples of the stationary distribution of stochastic matching models with reneging, in the case where arrival times are discrete. The first algorithm simply relies on the control of the model at hand by an infinite server queue, an algorithm that would clearly not be optimal in a context of heavy traffic. Indeed, as was observed in \cite{BD15}, as it relies on the depletion of a corresponding infinite-server model, the coalescence time for Algorithm \ref{algo2} grows exponentially in function of the arrival rates, see \cite{kelly91}. Our second algorithm, Algorithm \ref{algo3}, is peculiar to the case where patience times are deterministic (and so the matching policies {\sc fcfm} and {\sc edf} are equivalent). In that case, we propose an ad-hoc control of the system simulated backwards in time by the input of the system.  Then, the algorithm substantially reduces the number of operations compared to the primitive CFTP. In particular, if latency is allowed, we show that Algorithm \ref{algo3} also outperforms the algorithm based on the control by the infinite server queue, Algorithm \ref{algo2}. In fact, both Algorithms \ref{algo2} and \ref{algo3} can be seen as particular cases of a more general perfect sampling algorithm for bounded Markov chains, Algorithm \ref{algo1} below, which we call perfect sampling {\em by control}, a  condition that is closely related to those under which a DCFTP-type algorithm can be implemented. 

This paper is organized as follows. After some preliminary in Section \ref{sec:prelim}, we introduce our general perfect sampling algorithm by control in Section \ref{sec:algo}. In Section \ref{sec:matching}, we introduce the general stochastic matching model with reneging, and the two corresponding perfect sampling algorithms in sub-sections \ref{subsec:firstalgo} and \ref{subsubsec:secondalgo}. The performances of the latter algorithm are investigated in sub-section \ref{subsubsec:effi}. We compare the performance of Algorithm \ref{algo3} to the primitive CFTP algorithm in sub-section \ref{subsubsec:compare}, and to Algorithm \ref{algo2} in sub-section \ref{subsec:latency}, for a model with reneging and latency. A first application to the comparison of the steady-state performances of two matching policy (here, {\sc edf} (or in other words {\sc fcfm}) and {\sc ml}), is provided in sub-section \ref{subsubsec:compare}.

\section{Preliminary}
\label{sec:prelim}

In what follows, $\R$, $\N$, $\N^*$ and $\Z$ denote the sets of real, non-negative integers, strictly positive integers and relative integers, respectively. For any two elements $a,b\in \Z$, let $\llbracket a,b \rrbracket$ denote the integer interval 
$[a,b]\cap \Z$. 

Any (simple, finite and undirected) graph $G$ is denoted by $G=(\V,E)$, where $\V$ is the set of nodes and $E$ is the set of edges for a node $i\in \V$. For $n\in\N^*$, we say that $G$ is of size $n$ if the cardinality $|\V|$ of $\V$ is $n$. For any nodes $i,j\in \V$, we write $i - j$ if $i$ share an edge in $G$, that is, 
$\{i,j\}\in E$. Else, we write $i \nleftrightline j$. For any set $U\subset \V$, we denote by $E(U)$ the neighborhood of $U$, namely, 
\[E(U)=\left\{j\in \V\,:\, i - j \mbox{ for some }i \in U\right\}.\]
For simplicity, for all $i\in \V$ we set $E(i):=E(\{i\})$, the set of neighbors of node $i$ in $\V$.

Throughout the paper, all considered random variables (r.v.'s, for short) are defined on a common probability space 
 $(\Omega,\mathcal F,\P)$. 
 
\begin{definition}\rm 
Let $\X$ and $\mathbb{V}$ be two separable metric spaces. Let $k\in\Z$ and $x\in \X$. Let $f$ be a measurable mapping from $\,\X \times \mathbb{V}$ to $\X$, and $\suitez{v_n}$ be an identically distributed sequence of $\mathbb{V}$-valued r.v.'s. 
We denote by $\suitek{X_n^{k}(x)}$, the stochastic recursive sequence 
(SRS) driven by $\left(f,\suitez{v_n}\right)$, of initial value $x$ at time $k$. Namely, $\suitek{X_n^{k}(x)}$ is fully determined by the recurrence 
equation 
\[\begin{cases}
X^k_k(x) &=x\,;\\
X^k_{n+1}(x) &= f (X^k_n(x),v_n),\,\quad \mbox{a.s. for all }n\ge k.\end{cases}\]
\end{definition}
It is immediate that $\suitez{X_n^{k}(x)}$ is a $\X$-valued Markov chain whenever the sequence 
$\suitez{v_n}$ is IID. Conversely, any $\X$-valued Markov chain 
$\suite{Z_n}$ of fixed starting time $k$ and initial value $x$, having transition matrix $Q$ over $\X$, can be 
represented by the SRS
driven by $\left(f,\suitez{v_n}\right)$, where $\suitez{v_n}$ is an IID sequence of uniformly distributed r.v.'s on $[0,1]$ and $f$ is piecewise constant and satisfies for all $x_1,x_2\in \X$, 
$$ \lambda\left(\lbrace x\,:\,,f(x_1,x) = x_2 \rbrace\right) = Q(x_1,x_2),$$  
for $\lambda$ the Lebesgue measure, see e.g. Section 2.5.3 of \cite{BacBre02}.  

\medskip

Fix an SRS $X:=\suitek{X_n^{k}(x)}$.  
Then for all $e\in \X$, we set  
\[\tau^{X,k}_e(x) =\inf\left\{n\ge k\,:\, X_n^{k}(x)=e\right\},\]
the hitting time of value $e$ by $\suitek{X_n^{k}(x)}$. 
If $\suitez{v_n}$ is IID, then $\suitek{X_n^{k}(x)}$ is a Markov chain, and the distribution of 
$\tau^{X,k}_e(x) - k$ is independent of $k$. In that case, we then denote by $\tau^{X}_e(x)$, a generic r.v. that is 
so distributed.

\section{A perfect sampling algorithm by control}
\label{sec:algo}

In this section we present a perfect simulation algorithm, Algorithm \ref{algo1},  for processes that are bounded, in a sense that we will make precise hereafter. Our procedure is closely related to that of the dominated coupling from the past algorithm introduced by Kendall and Moller in \cite{Ken98,KM00}. 
Algorithm \ref{algo1} roughly proceeds as follows: simulating from the past an auxiliary chain $Y$, until at has reached one of the {\em end points}, at which time we start simulating a trajectory of the CTMC $X$, up to time 0. 
As will be shown hereafter, under certain conditions the output of Algorithm \ref{algo1} is sampled exactly from the stationary distribution of $X$. 

Until the end of this section, we fix three separable metric spaces $\X,\Y$ and $\mathbb{V}$, and two mappings $f:\X \times \mathbb{V}\to \X$ and $g:\Y \times \mathbb{V}\to \Y$. 
\medskip


\begin{algorithm}
\caption{Simulation of the stationary probability of $X$}
\KwData {$a_1,...,a_q \in \X$, $b_1,...,b_q,y\in \Y$, a probability distribution $\mu$ on $\mathbb{V}$}
$T_{down} \gets -1$    \;
$T_{up} \gets -1$ \tcc*[l]{We initialize the starting time.}
$Y \gets y$\;
\While{$Y \not\in \{b_1,...,b_q\}$}
 {$i \gets T_{up}$ \;
 $Y \gets y$\;
\For{$j \gets T_{up} \ \KwTo \ T_{down}$}
{ \bfseries{draw} $v_j$  \bfseries{from} $\mu$\tcc*[l]{We draw the random variables still needed for this iteration}}
\While{$i<0$ and $Y \not\in \{b_1,...,b_q\}$}
{$Y \gets g(Y,v_i)$ \;
 $i \gets i+1$ \;
 }
 $T_{down} \gets T_{up} - 1 $ \;
 $T_{up} \gets 2  T_{up}$
}
\For{$k \gets 1 \ \KwTo \ q$}
{\If{$Y = b_k$}
{$X \gets a_k $ \tcc*[l]{We assign to $X$ the state corresponding to the endpoint reached by $Y$.}}} 
\tcc{ We now transition $X$ to time $0$.}
\While{$i<0$ } 
{$X \gets f(X,v_i)$ \;
$i \gets i+1 $ \;}
\KwRet{$X$}
\label{algo1}
\end{algorithm}

\subsection{A control condition}
The {\em control} of an SRS of interest by an auxiliary one, is the key to our perfect simulation algorithm. It is defined hereafter, 
\begin{definition}
\label{def:control}
Let $\suitez{X_n}$ and $\suitez{Y_n}$ be two SRS, respectively valued in $\X$ and $\Y$ and $q \in \N^*$. We say that $\suitez{Y_n}$ {\em $q$-controls} $\suitez{X_n}$, if there exists 
$b_1,...,b_q,y \in \Y$ and $a_1,...,a_q \in \X$ such that  
\begin{equation}
\label{eq:majoration}
\forall i \in \llbracket 1, q \rrbracket, \forall k \in \Z,\,\forall n \ge k,\,\left[ Y^{k}_n(y)=b_i \right] \Longrightarrow \left[\forall x \in \X,\, X^k_n(x)=a_i \right].
\end{equation}
$b_1,...,b_q$ are called the {\em endpoints} of $Y$. If  $q = 1$ we simply say that $Y$ {\em controls} $X$.
\end{definition}
The following result establishes that under certain conditions 
including the control of $\suitez{X_n}$ and $\suitez{Y_n}$, Algorithm \ref{algo1} terminates almost surely, and the output 
is a sample of the stationary distribution of the SRS $\suitez{X_n}$.

\begin{theoreme}
\label{thm:main}
Suppose that the sequence $\suitez{v_n}$ is IID, and let $X$ and $Y$ be two SRS respectively driven by 
$\left(f,\suitez{v_n}\right)$ and $\left(g,\suitez{v_n}\right)$. Suppose that 
$X$ is $q$-controlled by $Y$ for $b_1,...,b_q,y$ and $a_1,...,a_q$, and that it holds that
\begin{equation}
\label{eq:recurY}
\pr{\tau^Y_{b_i}(y) <\infty} =1, \quad i \in \llbracket 1, q  \rrbracket. 
\end{equation}
Then Algorithm \ref{algo1} terminates almost surely, and its output is sampled from the unique stationary distribution of $X$. 
\end{theoreme}

\begin{proof}[Proof]\rm
We first show that Algorithm \ref{algo1} terminates almost surely. To see this, observe that 
for any $i \in \llbracket 1, q \rrbracket$ 
and any $N \in \N $, 
\begin{align*}
\Pb\left(\bigcup\limits_{n\in \N } \left\lbrace \tau^{Y,-2^n}_{b_i}(y)\leq 0\right\rbrace \right)
&=\Pb\left(\bigcup\limits_{n\in \N } \left\lbrace \tau^{Y,-2^n}_{b_i}(y)+2^n\leq 2^{n} \right\rbrace \right)\\
&\geq \Pb\left(\left\lbrace \tau^{Y,-2^N}_{b_i}(y)+2^N\leq 2^{N} \right\rbrace\right)
    =  \Pb\left( \left\lbrace \tau^{Y}_{b_i}(y) \leq 2^{N}  \right\rbrace\right),  
\end{align*}
in view of the stationarity of the input. 
So we obtain that 
\begin{align*}
   \Pb\left(\bigcup\limits_{n\in \N } \left\lbrace \tau^{Y,-2^n}_{b_i}(y)\leq 0\right\rbrace \right)
   &\geq \lim_{N \rightarrow +\infty} \ \Pb\left( \left\lbrace \tau^{Y}_{b_i}(y) \leq 2^{N}  \right\rbrace\right)\\
    &= \Pb( \lbrace \tau^{Y}_{b_i}(y) < +\infty  \rbrace) = 1,
\end{align*}
showing that Algorithm \ref{algo1} terminates almost surely. 

Now, let $N$ be the backwards coalescence time of the chain $X$, 
that is, the smallest starting time for which the CFTP algorithm terminates for $X$, or in other words  
\begin{equation}
\label{eq:defcoalesce}
N=\min\left\{n\ge 0\,:\,X^{-n}_0(x)=X^{-n}_0(x')\mbox{ for all }x\ne x' \in \X\right\}.
\end{equation}
Let $R$ be the smallest termination time of Algorithm \ref{algo1}. Then, by the very 
definition of Algorithm \ref{algo1} and (\ref{eq:majoration}) there exists $ i \in \llbracket 1, q \rrbracket$, a time $n_0>0$ such that $-R<-n_0$, and such that 
$X^{-R}_{-n_0}(x)=X^{-R}_{-n_0}(x')=a_i$ for all $x,x'\in \X$, $x\ne x'$, and thereby 
\begin{equation}
\label{eq:coalesce}
X^{-R}_0(x)=X^{-R}_0(x')=X^{-n_0}_0(a_i),\mbox{ for all $x,x'\in \X$, $x\ne x'$. }
\end{equation}
In particular, it naturally follows from (\ref{eq:defcoalesce}) that we necessarily have 
$R\ge N$, otherwise all versions of the chain $X$ starting from a time posterior to $-N$, would have coalesced before 0, an absurdity. 
In particular, $N$ is almost surely finite. In the terms of \cite{FT98}, the vertical backwards coalescence time of $X$ is successful, and so it follows from Theorem 4.1 in \cite{FT98}, first, that there exists a unique invariant probability $\pi$ for the chain $X$, and second, that 
the output $X^{-N}_0(x)$ of the CFTP algorithm when started from any $x\in\X$, is sample from $\pi$. 
But it also follows from \eqref{eq:coalesce}, that $$X^{-R}_0(x)=X^{-N}_0(x)=X^{-n_0}_0(a_i),\quad \mbox{ for all }x\in \X.$$ 
So Algorithm \ref{algo1} and the CFTP algorithm produce the same output, which completes the proof. 
\end{proof}

\begin{remark}
The assumptions of Theorem \ref{thm:main} are satisfied in particular if $Y$ is positive recurrent and irreducible on the discrete state space $\Y$, or if the distribution of $Y$ has atoms at points $b_1,...,b_q$ with finite hitting times from $y$.  
\end{remark}

\begin{remark}
It follows from the equivalence shown in Theorem 4.2 of \cite{FT98} that under the assumptions of Theorem \ref{thm:main}, the Markov chain 
$X$ is uniformly ergodic, since the vertical coalescence time for $X$ is successful. 
\end{remark}

\subsection{Renovating events and small sets}
\label{subsec:lienrenove}
Assumption \eqref{eq:majoration} is key to our analysis. Under this control condition the value of the SRS $Y$ forces that of $X$ at time $n$, whatever the value of $X$ at time $N$.
This is reminiscent of the concept of renovating event, as introduced by Borovkov and Foss, see \cite{BF92,BF94}. 
Let us remind the following, 
\begin{definition}
Let $X$ be an SRS driven by $f$ and $\suitez{v_n}$, and $\suite{A_n}$ be a sequence of events.
We say that $\suite{A_n}$ is a sequence of {\em renovating events} of length $m$ and associated mapping 
$h: \mathbb{V}^m\rightarrow \X$ for the chain $X$  if for any $n \in \Z$, on $A_n$ we have 
$$X_{n+m} = h(v_n,...,v_{n+m-1}).$$
\end{definition}

Now suppose that \eqref{eq:majoration} holds for $a_1,...,a_q,b_1,...,b_q$ and $y$. Then, it is easily seen that for all $k\in\Z$, $ i \in \llbracket 1, q \rrbracket$ and 
$x\in\X$, $\suitek{\{Y^{k}_n(y)=b_i\}}$ is a sequence of renovating events of length 1 for the sequence $\suitek{X^k_n(x)}$. Indeed, for all 
$n\ge k$, on $\{Y^{k}_n(y)=b_i\}$ we get that $X^k_n(x)=a_i$, and therefore $$X^k_{n+1}(x)=f(a_i,v_n)=:h(v_n).$$ 
Then, various conditions on the events $\suitek{\{Y^{k}_n(y)=b_i\}}$ can be given, that imply that there exist a stationary version of the chain 
$\suite{X_n}$, see e.g. Theorem 2.5.3 and Property 2.5.5 in \cite{BacBre02}. 

There is also an insightful connection between the control condition of Definition \ref{def:control} and the concept of small set, that is recalled hereafter under the formulation of Chapter 5 of \cite{MT12}. 
%
\begin{definition}
\label{def:small}
For a positive integer $m$, we say that the subset $A\subset \X$ is $m$-{\em small} for the Markov chain $\suite{X_n}$ on $\X$, if there exists $\eta_m>0$, and a non-null Borel measure $\mu_m$ on $\X$ such that 
\[\forall x \in A,\forall B\in \mathcal B(\X),\quad \pr{X_m \in B \,|\, X_0 = x} \ge \eta_m\mu_m(B).\]
\end{definition}
Thus, starting from such set, the chain (partially) regenerates in a finite horizon of size $m$, since after that, the transitions of the chain do not depend on the starting point $x$ with strictly positive probability. The existence of small sets is of crucial use in the construction of uniformly ergodic Markov chains, see \cite{FT98,ken04,CK07}. 

It is significant that under the control condition of Definition \ref{def:control} and \eqref{eq:recurY}, the whole set $\X$ is small. 
To see this, observe that for any $i\in\llbracket 1,q \rrbracket$, for any $m\in\N$ such that $\pr{\tau^Y_{b_i}(y)=m-1}>0$, in view of the Markov property, for all $x\in \X$ and all borelian subsets $B\subset \X$ we have 
\begin{align*}
\pr{X_m \in B \,|\, X_0 = x}
&\ge \pr{\{X_m \in B\}\cap\{X_{m-1}=a_i\}\,|\, X_0 = x}\\
					&=\pr{X_{m-1}=a_i\,|\,X_0 = x}\pr{X_m \in B\,|\,\{X_{m-1}=a_i\} \cap\{X_0 = x\}}\\
					&\ge \pr{\tau^Y_{b_i}(y)=m-1}\pr{X_1 \in B\,|\,X_{0} = a_i}\\
					&=: \delta^i_m \mu^i_m(B).
					\end{align*}
Hence $\X$ is $m$-small.  

\subsection{The ordered case} 
\label{subsec:ordre}
A typical context in which the control of the SRS $X$ by the SRS $Y$ occurs, is when the two sequences are constructed on the same 
input, and their driving maps satisfy some monotonicity properties, which we detail below. 
Throughout this section, $(\mathbb U,\prec)$ denotes a partially ordered space, 
and we define two mappings $\varphi:\,\X \longmapsto \mathbb U$ and $\psi:\,\Y \longmapsto \mathbb U$. 
 
\begin{definition}\rm
We say that the mapping $f: \X \to \X$ is dominated (for $\mathbb U$, $\varphi$ and $\psi$) by the mapping $g :\Y \to \Y $, 
and denote $f \prec^{\mathbb U,\varphi,\psi} g$, if 
\begin{equation*}
\forall x \in \X,\,y\in\Y, \quad \left[\varphi(x)\prec \psi(y)\right] \Longrightarrow \,\left[\varphi\circ f(x) \prec \psi\circ g(y)\right].
\end{equation*} 
\end{definition}
\noindent In the definition above, $\mathbb U$ is an auxiliary partially ordered set that is used for comparing $f$ to $g$ via the projections 
$\varphi$ and $\psi$. Observe the following simple particular case, 
\begin{proposition}
\label{prop:equivord}
In the case where $\X=\Y=\mathbb U$, $\X$ is partially ordered by $\prec$ and $\varphi=\psi=\mbox{i}$ the identity function, we have $f\prec^{\X,\mbox{i},\mbox{i}} g$ under either one of the conditions below, 
\begin{itemize}
\item[(i)] $g$ is $\prec$-nondecreasing and pointwise lower-bounded by 
$f$;
 \item[(ii)] $f$ is $\prec$-nondecreasing and pointwise upper-bounded by 
$g$. 
\end{itemize}
\end{proposition}
\begin{proof}[Proof]
Plainly, for all $x,y \in \X$ such that $x\prec y$, if we assume that (i) holds, then we get 
$f(x)\prec g(x)\prec g(y),$    
whereas if (ii) holds we obtain that 
$f(x) \prec  f(y)\prec g(y).$ 
\end{proof}
\begin{proposition}
\label{prop:simuord}
Let $X$ and $Y$ be two SRS respectively driven by $(f,\suitez{v_n})$ and 
$(g,\suitez{v_n})$, where the input $\suitez{v_n}$ is  
IID on $\mathbb{V}$. Suppose that $f(.,v)\prec^{\mathbb U,\varphi,\psi} g(.,v)$ for all $v\in \mathbb{V}$, where $\mathbb U$ admits 
the $\prec$-minimal point $o$. Suppose also that $\varphi^{-1}(o) =\{a\}$, that there exists $y\in\Y$ such that 
\begin{equation}
\label{eq:condy}
\forall x\in \X,\,\,\varphi(x) \prec \psi(y),
\end{equation} 
and that $\tau^Y_b(y)$ is almost surely finite for some $b\in \psi^{-1}(o)$. 
Then, Algorithm \ref{algo1} for $y$, $a$ and $b$, terminates a.s., and produces a sample of the unique stationary distribution of $X$.
\end{proposition}
\begin{proof}[Proof]
We aim at showing that $X$ controls $Y$ for $b,y,a$. 
Let $k,n\in\Z$ be such that $n>k$ and $Y^{k}_n(y) = b$. Let $ x \in \X $. We show by induction on $\ell$, that for all $\ell\in\llbracket k,n \rrbracket$, 
\begin{equation}
\label{eq:recurord}
\varphi(X^k_{\ell}(x)) \prec \psi(Y^k_{\ell}(y)).
\end{equation}
First, from (\ref{eq:condy}) we get that 
$ \varphi(X^{k}_k(x))= \varphi (x) \prec \psi(y) = \psi(Y^{k}_k(y)),$ 
so (\ref{eq:recurord}) holds for $\ell=k$. Suppose that it is true at rank $\ell \in \llbracket k,n-1 \rrbracket$, i.e.,  
that $\varphi(X^k_{\ell}(x)) \prec \psi(Y^k_{\ell}(y))$. 
Then, 
from the domination assumption of $f$ by $g$ we obtain that 
$$\varphi(X^k_{\ell+1}(x)) = \varphi(f(X^k_{\ell}(x),v_{\ell})) \prec \psi(g(Y^k_{\ell}(y),v_{\ell})) = \psi(Y^k_{\ell+1}(y)),$$ 
so (\ref{eq:recurord}) holds at rank $\ell+1$. It is therefore true for all $\ell\in\llbracket k,n \rrbracket$. 
In particular, we have that $$\varphi(X^{k}_n(x))\prec \psi(Y^{k}_n(y)) = \psi(b) =o,$$
implying that $X^{k}_n(x) = a $. Thus $Y$ controls $X$, and we conclude using Theorem \ref{thm:main}. 
\end{proof}

As a conclusion, provided that $f\prec^{\mathbb U,\varphi,\psi} g$, Algorithm \ref{algo1} provides a perfect sampling algorithm for the SRS $X$. 
In fact, in this ordered case, Algorithm \ref{algo1} is closely related to the DCFTP algorithms of Kendall, see \cite{ken04,CK07}. Specifically, as in \cite{CK07}, thanks to \eqref{eq:condy} we have an upper bound process $Y$, that we can simulate backwards in time. We also have a lower bound process, namely the constant process equal to $b$. Similarly to the sandwiching method in \cite{PW96}, we only have to simulate the process $Y$ starting at state $y$. When that process meets the lower bound backwards in time, means that coalescence has been detected. Then, as in \cite{ken04,CK07}, we simulate $X$ starting from a single state until time $0$. Notice that the DCFTP algorithms introduced in \cite{ken04} and \cite{CK07} use geometric ergodicity, and 
are based on small sets constructions. As we observed above (Section \ref{subsec:lienrenove}), 
the control condition implies the smallness of $\X$, so our approach is reminiscent of this idea. 

Observe that similar approaches are used for the perfect sampling of loss queueing systems in \cite{BD15} using the domination of the system by an infinite server queue in some sense (an idea that we also use in the construction of Section \ref{subsec:firstalgo} below) and likewise, for the perfect sampling of multiple-server queues in \cite{BD18}. 

\begin{remark}
The above DCFTP conditions are in fact reminiscent of stochastic domination conditions 
for the construction of stationary SRS's in the general stationary ergodic context. For instance, for $\X=E$ a lattice space, Condition (i) in Proposition \ref{prop:equivord} above amounts to condition (H1) in \cite{moyal2015} for any SRS $X$ and $Y$ that are respectively driven by $f$ and $g$, and a common input $\suitez{v_n}$. This latter condition guarantees, under general stationary ergodic assumptions, the existence of a stationary version of the SRS $X$, at least on an extended probability space, provided that a stationary version of the SRS $Y$ exists on the original one. See \cite{lisek1982}, and Theorem 3 in \cite{moyal2015}. 
\end{remark}

\newpage 

\section{A stochastic matching model with impatience}
\label{sec:matching}
In this section we address the perfect sampling of the stationary state of 
a class of models, which we refer to as `general stochastic matching models with impatience'. 

\subsection{The model}
\label{subsec:model}
 We consider a general stochastic matching model (GM), as was
defined in \cite{mairesse2016stability}: items enter one by one in a system, and each of them belongs to a determinate
class. The set of classes is denoted by $\mathbb{V}$, and identified with $\llbracket 1,|\mathbb{V}|\rrbracket$. 
We fix a simple, connected graph $G = (\mathbb{V}, E)$ having set of nodes $\mathbb{V}$, termed {\em compatibility graph}. Upon arrival, any incoming item of class, say, 
$i \in  \mathbb{V}$ is either matched with an item present in the buffer,
of a class $j$ such that $i$ shares an edge with $j$ in $G$, if any, or if no such item is available, it is stored in the buffer to wait for its match. 
Whenever several possible matches are possible for an incoming item $i$, a {\em matching policy} determines what is the match of $i$ without ambiguity. 
Each matched pair departs the system right away.

A GM model with impatience is a GM model in which each entering item in the system is assigned a patience time upon arrival. If the considered item has not been matched at the end of her patience time, then she leaves the system forever. 
%
To formalize this, after fixing the compatibility graph $G=(\mathbb{V},E)$ and the matching policy $\Phi$, we consider that arrivals occur at integer times, i.e., we suppose that the generic inter-arrival time $\xi$ is constant equal to one, and fix 
 two IID sequences $\suitez{V_n}$ and $\suitez{P_n}$, where for all 
$n\in\Z$, $P_n\in\R_+$ and $V_n\in \mathbb{V}$ respectively represent the patience time and the class of the $n$-th item entering the system. 
We denote respectively by $V$ and $P$, generic r.v.'s distributed like $\suitez{V_n}$ and $\suitez{P_n}$ respectively, and assume throughout that the r.v. $P$ is integrable. The two sequences $\suitez{V_n}$ and $\suitez{P_n}$ are not necessarily independent. 
In particular, it can be the case that the patience time $P_n$ of the $n$-th item depends on her class $V_n$. 
In what follows, we denote by $\mu$ the law of $V$ on $\V$. 

The class of models defined in Section \ref{subsec:model} admits the following Markov representation. 
Define the set 
\[{\mathbb X}:=\{\emptyset\}\cup \bigcup_{q=1}^{\infty} \left(\R^*_+\times \mathbb{V}\right)^q.\]
For all $t \ge 0$, let $Q(t)$ be the number of customers in the system at time $t$, 
and let us define the {\em profile} of the system at $t$, as the 
following element of ${\mathbb X}$, 
\begin{equation}
\label{eq:defX}
 X(t) = \begin{cases}
          \left( \left({R}^1(t),V^1(t)\right),\cdots, \left({R}^{{Q}(t)}(t),V^{Q(t)}(t)\right)\right)&\mbox{ if }{Q}(t)\ge 1,\\
          \emptyset &\mbox{ else,}
          \end{cases}
\end{equation}
where for all $i\in \llbracket 1,Q(t) \rrbracket$, 
we denote by $R^i(t)$ (resp., $\mathbb{V}^i(t)$) the remaining patience at time $t$ (resp., the class) 
of the $i$-th item in line at time $t$, in the order of arrivals. If the system is empty at $t$, we again set $ X(t)=\emptyset$. 

\begin{definition}
We say that the matching policy $\Phi$ is {\em admissible} if, upon each arrival, the choice of the match amongst compatible items 
in line at $t$, if any, is made according to the sole knowledge of $ X(t)$, and possibly of a draw that is independent of everything else. 
\end{definition}

\begin{remark}
It is easily seen that matching policies that depend only on the arrival times ({\em First Come, First Matched}, denoted hereafter by {\sc fcfm}, or {\em Last Come, First Matched}), remaining patience times ({\em Earliest Deadline First}, {\em Latest Deadline First}), matching policies that depend on the queue sizes of the various nodes ({\em Match the Longest}, {\em Match the Shortest}, {\em Max-Weight}) and priority policies are all admissible. See e.g. \cite{mairesse2016stability,MBM21,jonckheere2020generalized} for a detailed presentation of admissible policies for classical matching models. 
\end{remark}
Set $\suitez{T_n}=\suitez{n}$, the arrival times to the system, and for all $n\in\Z$, denote 
by $X_n= X(T_n^-)=X(n^-)$, the state of the system seen by the customer entered at time $n$.  
Then we obtain the following result, 

\begin{proposition}
\label{prop:SpSmatching}
For any admissible matching policy $\Phi$, the profile sequence 
$\suitez{ X_n}$ is stochastic recursive, 
driven by the couple sequence  $\suitez{(V_n,P_n)}$, and 
a mapping $ f^\Phi:\X\times (\R_+\times \mathbb{V}) \longmapsto \X$ that depends on $\Phi$ and possibly on a random draw independent of everything else. 
In other words we get that 
\[ X_{n+1}= f^\Phi\left( X_n,(P_n,V_n)\right),\quad n\in\Z.\]
\end{proposition}
\begin{proof}
The construction of $f^\Phi$ is immediate: If the incoming element at $n$ is matched upon arrival, 
the couple corresponding to its match, determined by $\Phi$, is erased from the vector $X_n$; else, the couple $(V_n,P_n)$ is added 
at the end of the vector $X_n$. Last, the couples (possibly including the incoming couple $(V_n,P_n)$) whose second coordinate is strictly less than $1$ at $n$ are erased from the vector $X_n$ (because they will have reneged by time $n+1$), and the second coordinates of all other couples of $X_n$, if any, decrease by $1$.  
\end{proof}

\subsection{A first Perfect sampling algorithm}
\label{subsec:firstalgo} 
We can then design a first perfect sampling algorithm for matching models with impatience, that is simply based on the control (in the sense of Section \ref{sec:algo}) by an infinite server system. 
In this context, we let $ Y:=\suitez{Y_n}$ be a $\R_+$-valued SRS defined by the recursion 
\begin{equation}
\label{eq:recurYtilde}
Y_{n+1}=\left[\max(Y_n,P_n)-1\right]^+=: g\left(Y_n,P_n\right),\quad n \in \Z. 
\end{equation}
Then, for all $n$, $Y_n$ can be interpreted as the largest remaining service time of a 
D/GI/$\infty$ queue of service times $\suitez{P_n}$, upon the arrival of the $n$-th customer. 
As the generic r.v. $P$ is assumed integrable, it is well known that whenever 
\begin{equation}
\label{eq:condstabmatching}
\Pb(P \le \xi) = \Pb(P \le 1) >0
\end{equation} 
the Markov chain $\suitez{Y_n}$ is positive recurrent: See e.g. Corollary 4.32 in \cite{DM12}, \cite{Tho00}, and the generalization to the case where $\suitez{P_n}$ is stationary ergodic, combining Lemma 5 of \cite{moyal2008stability} with Corollary 2 in \cite{moyal2013queues}. 

Consider Algorithm \ref{algo2}, 
which is a declination of Algorithm $\ref{algo1}$ started with $y=m$ for $m$ defined below, for $Y$ the recursion defined by \eqref{eq:recurYtilde}, $q=1$, $a_1=\emptyset$ and $b_1=0$.

\begin{algorithm}
\caption{Simulation of the stationary probability of $ X$ - Matching model with impatience}
\KwData {A probability distribution $\mu$ on $\mathbb{V} \times \R^+$} 
 $T_{down} \gets -1$ \;
 $T_{up} \gets -1$ \tcc*[l]{We initialize the starting time.}
 $ Y \gets m$ \;
\While{$ Y \neq 0$}
{ $i \gets T_{up}$ \;
 $Y \gets \emptyset$ \;
\For{$j \gets T_{up} \ \KwTo \ T_{down} $}
{{ \bfseries{draw} $(v_j,p_j)$  \bfseries{from} $\mu$\tcc*[l]{We draw the random variables still needed at this iteration}}}
\While{$i<0$ and $ Y \neq 0$}
{ $ Y \gets  \left[\max(Y,p_i)-1\right]^+$ \;
$i \gets i+1$ \;}
 $T_{down} \gets T_{up}-1$ \;
$T_{up} \gets 2T_{up}$ \;
}
$X \gets \emptyset$ \tcc*[l]{As $Y$ has reached $0$ we know that $X$ has reached $\emptyset$. } 
\tcc{ We now transition $ X$ to time $0$ as a matching system.}
\While{$i<0$}
{$ X \gets  f^\Phi(X,(v_i,p_i))$ \;
$i \gets i+1$ \;}
\KwRet{$X$} 
\label{algo2}
\end{algorithm}

\begin{theoreme}
\label{thm:perfectmatching}
Under condition (\ref{eq:condstabmatching}), the profile Markov chain $\suitez{ X_n}$ admits a unique stationary distribution. 
If moreover there exists $m>0$ such that  
$\Pb(P\leq m)=1$, then Algorithm $\ref{algo2}$ terminates almost surely, and its output is sampled from the stationary distribution of $\suitez{ X_n}$. 
\end{theoreme}

\begin{proof}[Proof]
We apply Proposition \ref{prop:simuord}, by setting in this case 
\begin{eqnarray}
\label{varphideux}
\varphi: & \X &\longrightarrow \R_+\\
 & x=\left((r_1,v_1),\cdots,(r_q,v_q)\right)\ne \emptyset &\longmapsto \max\left\{r_i:\,i\in\llbracket 1,q \rrbracket\right\}\notag\\
 &\emptyset &\longmapsto 0.\notag
 \end{eqnarray}
As any item spends in the system a time that is less or equal to 
its patience time, for any $n\in\Z$, 
$\varphi( X_n)$ corresponds to the largest remaining maximal sojourn time in the system of an item in the system just before time $T_n$. 
Consequently, for any $(p,v) \in \R_+\times \mathbb{V}$, for all $ x\in \X$ we obtain that 
\begin{equation}
\label{eq:majrecPhi}
\varphi\left( f^\Phi\left( x,(p,v)\right)\right) \le \left[\max\left(\varphi\left( x\right),p\right) - 1\right]^+=  g\left(\varphi\left( x\right),p\right).
\end{equation}
Therefore, for any $ x\in\X$ and $y\in\R_+$ such that $\varphi( x) \le y$, for any $(p,v)$, as $ g(.,p)$ is non-decreasing on $\R_+$
we get that \[\varphi\left( f^\Phi\left( x,(p,v)\right)\right) \le  g\left(y,p\right).\]
Proposition \ref{prop:simuord} completes the proof. 
\end{proof}

\subsection{Deterministic patience times} \label{subsec:detpatience}
Whenever condition (\ref{eq:condstabmatching}) does not hold, the existence and uniqueness of a stationary distribution for the Markov chain $\suite{X_n}$ are not granted. One then has to resort to ad-hoc techniques to show stability and to sample the stationary state. 

In this section, we consider the particular 
case of the previous model, in which patience times are deterministic. Specifically, we suppose that $P\equiv p+\varepsilon$, for some $p\in\N^*$ and $0<\varepsilon<1$. Assuming that patience times are not integers, while arrivals occur at integer times, allows us to avoid the ambiguous situation in which an element enters the system and finds an element of remaining patience times zero. 
In practice, any incoming element can either be matched upon arrival, or with any of the $p$ following entering items. If not, the item is lost before the arrival of the $p+1$-th element after it, because its remaining time then equals $\varepsilon-1<0$. For short, we denote such a matching model by $(G,\Phi,\mu,p)$.

Clearly, in this context, (\ref{eq:condstabmatching}) fails. (Notice that taking $p=0$ in the present construction would lead to a system in which no item could ever be matched.)  
In this section, we show that such systems are nevertheless positive recurrent, and construct an alternative perfect sampling algorithm 
that is another declination of Algorithm \ref{algo1}, and is again based on the control condition defined in Section \ref{sec:algo}. 

\subsubsection{Alternative Markov representation}
In this particular case, the profile Markov chain can be simplified, so as to obtain the following alternate, simpler, Markov representation 
of the system state, 


\begin{definition}
For all $n\in\Z$, the {\em word-profile} of the system just before time $n$ is defined by the word 
$$\tilde X_n = w_1\cdots w_p \in (\mathbb{V}\cup\{0\})^p,$$
 where for all $i \in \llbracket 1,p \rrbracket$, 
\[w_ i = \begin{cases}
         V_{n-p+i-1}  & \mbox{if the item entered at }n-p+i-1\mbox{ was not matched before }\small{n};\\
          0 &\mbox{else}.    
          \end{cases}.\]
\end{definition}   
In particular, if the element entered at time $n-p$ is still in the system at time $n$ (its class thus appearing as the first letter of the word $\tilde X_n$), it is either matched with the incoming element at time $n$, or it is considered lost.



We call $\XX \subset (\mathbb{V}\cup\{0\})^p$, the (finite) state space of $\tilde X$. 
Similarly to Proposition \ref{prop:SpSmatching}, it is immediate that for any admissible policy $\Phi$, the sequence $\suitez{\tilde X_n}$ is a Markov chain, and we denote by $\tilde f^{\Phi}$, the (deterministic, up to a possible draw that is independent of everything else) map $\tilde f^\Phi: \XX\times \mathbb{V} \to \XX$, such that  
\[\tilde X_{n+1}=\tilde f^\Phi\left(\tilde X_n,V_n\right),\quad n\in\Z.\]

 \subsubsection{Synchronizing words}
For a fixed model $(G,\Phi,\mu,p)$, with $G = (V,E)$, let $\mathbb{V}^*$ be the set of words on $\mathbb{V}$. For any word $v = v_1\cdots v_l$ in $\mathbb{V}^*$ and any $\tilde X\in \XX$, let us denote 
by $W^\Phi(\tilde X,v)\in\XX$, the state of a system started at $\tilde X$ and receiving the arrivals 
$v_1,...,v_l$ in that order. 

\begin{definition}
Fix a model $(G,\Phi,\mu,p)$. 
A  word $w = w_1\,\cdots\, w_q \in \mathbb{V}^*$ is said to be {\em synchronizing}, 
if 
$$\exists z(w) \in \XX:\, \forall \tilde x \in \XX,\,W^\Phi(\tilde x,w) = z(w).$$
\end{definition}
In other words, $w$ is a synchronizing word if all buffers synchronize to some value $z(w)$, whenever they are fed by a common arrival scenario $w$, whatever the initial state. It is obvious how synchronizing words can be used for perfect simulation. Indeed, if we start the Markov chain 
at a time $-M$ from all possible states, observing a synchronizing word of length $q<M$ amongst the arrivals (in the sense that the classes of $q$ consecutive incoming items are given by the letters of $w$, in that order), clearly guarantees that all chains have coalesced by time 0. 
In fact, recalling Definition \ref{def:small}, it is immediate that whenever there exists a synchronizing word $w$, then the whole 
set $\XX$  is $(q+1)$-small for the chain$\suitez{\tilde X_n}$. Indeed, for all $\tilde x\in\XX$ and $B\subset \XX$, 
\begin{align*}
\pr{\tilde X_{q+1} \in B \,|\, \tilde X_0 = \tilde x}
&\ge \pr{\{\tilde X_{q+1} \in B\}\cap\{V_0V_1\,\cdots\, V_{q-1}=w\}\,|\, \tilde X_0 = \tilde x}\\
					&=\pr{V_0V_1\,\,\cdots\,\, V_{q-1}=w}\pr{\tilde X_{q+1} \in B\,|\,\{V_0V_1\,\cdots\, V_{q-1}=w\}\cap\{\tilde X_0 = \tilde x\}}\\
					&= \prod_{i=0}^{q-1}\mu(w_i) \pr{\tilde X_{q+1} \in B\,|\,\tilde X_{q} = z}. 
					\end{align*}

In fact, our approach hereafter for perfect simulation is reminiscent of the small-set techniques for exact sampling in \cite{MG98,GM99,Wil00}. Specifically, we will use 
the arrivals of synchronizing words as a control to ensure the coalescence of all versions of the Markov chains. 

We first provide a sufficient condition for the existence of synchronizing words, for any discrete matching system. Hereafter, for any $k,\ell \in V$ we write $k\leftrightline \ell$ if $(k,\ell) \in E$, that is, the nodes $k$ and $\ell$ share an edge in $G$. Else, we write $k\nleftrightline \ell$. 

\begin{definition}
\label{def:strongsynchro}
Let $w\in \mathbb{V}^*$. We say that the word of length $2p$, 
$w=w_1\,\cdots\, w_{2p}\in \mathbb{V}^*$ is \em{strongly synchronizing}, if 
$$\forall i \in \llbracket 1,p \rrbracket,\,\forall j \in \llbracket p+1, p+i \rrbracket,\, w_i \nleftrightline w_j.$$
\end{definition}
\noindent The term {\em strongly synchronizing} is justified by the following result, 
\begin{theoreme}
\label{thm:sufficient}
In a discrete matching model with impatience $(G,\Phi,\mu,p)$, any strongly synchronizing word 
is a synchronizing word.
\end{theoreme}

\begin{proof}[Proof] Let $w=w_1 \,\cdots\, w_{2p}$ be a strongly synchronizing word, and let $u = w_1\,\cdots\, w_p$ and $v = w_{p+1}\,\cdots\, w_{2p}$. Let $\tilde x \in \XX$ and $\mathbf 0_p = \underbrace{0\,\cdots\, 0}_{p}$, be the empty state. 
As $u$ is of length $p$, any item present in the buffer represented by $\tilde x$ is no longer in there after the arrivals represented by $u$ (it is either matched or discarded, possibly just after the arrival of the last item of class $w_p$). Therefore $W^\Phi(\tilde x,u)= u'= w'_1,...,w'_p$ where for all $i \in \llbracket 1, p \rrbracket,  w'_i= w_i$ if the corresponding item is still in the buffer after these arrivals, or $w'_i=0$ else.  
As $w$ is strongly synchronizing, for any $i \in \llbracket 1, p \rrbracket$ such that $w'_i \ne 0$ and any $j \in \llbracket p+1; p+i \rrbracket$, we have that $w'_i \nleftrightline w_j$. All elements corresponding to non-zero letters of $u'$ are not matched, because their patience necessarily expires before the arrival of a compatible item, and no letters from $v$ can be married to a letter in $u'$. Therefore if $j,h \in \llbracket p+1, 2p \rrbracket$ are such that the element corresponding to $w_j$ is matched to that corresponding to $w_h$ if we add $v$ to the empty buffer $\mathbf 0_p$,  then it is also the case 
 if we add $v$ to the buffer $u'$. In other words, we get that 
$$W^{\Phi}(\tilde x,w) = W^\Phi(W^\Phi(\tilde x,u),v) = W^\Phi(u',v)=W^\Phi(\mathbf 0_p,v).$$
As this is true for any $\tilde x\in\XX$, $w$ is a synchronizing word for $z(w):=W^\Phi(\mathbf 0_p,v).$ 
\end{proof}

We proceed with two technical lemmas.  
In what follows, for all $a\in \mathbb{V}\cup \{0\}$ and all $k \in \llbracket 0 , p \rrbracket$ we define the following word of length $p$,
\[x^a(k)=\underbrace{0\cdots 0}_k\underbrace{a\cdots a}_{p-k}.\]
First observe the following, 
\begin{lemme}
\label{one letter}
Consider a matching model with impatience $(G,\textsc{fcfm},\mu,p)$, with matching policy {\sc fcfm}. Let $a \in \mathbb{V}$. 
Then, for all $k \in \llbracket 0 , p-1 \rrbracket$, for all words $w$ of length $p$, $W^\Phi(x^a(k),w)$ and $W^\Phi(x^a(k+1),w)$ differ at most by one letter in some position $i$ (substituting 0 to the $i$-th letter).
\end{lemme}
\begin{proof}[Proof]
Let $k \in \llbracket 0 , p-1 \rrbracket$, and write $w=w_1\,\cdots \,w_p$. With some abuse, in the proof 
below the matching procedure of the initial state $x^a(k)$ (or $x^a(k+1)$) with the arrival represented by $w$ is itself called $W^{\textsc{fcfm}}(x^a_{k},w)$ (or $W^{\textsc{fcfm}}(x^a(k+1),w)$). 

If $w_i \nleftrightline a$ for all $i\in\llbracket 1,p \rrbracket$, then we trivially get that 
$$W^{\textsc{fcfm}}(x^a(k),w)=W^{\textsc{fcfm}}(x^a(k+1),w).$$ Else, let $i_1,...,i_l$ be the indices, in increasing order, of the letters of $w$ matched with letters of $x^a(k+1)$ in $W^{\textsc{fcfm}}(x^a(k+1),w)$. 
There are three possibilities for the indices (in increasing order) of the letters of $w$ that are matched 
with letters of $x^a_{k}$ in $W^{\textsc{fcfm}}(x^a_{k},w)$ (which we call for short ``the indices'' in the discussion hereafter): 
\begin{enumerate}
\item Either the first $a$ of $x^a_{k}$ is matched in $W^{\textsc{fcfm}}(x^a_{k},w)$ with a letter of $w$ 
of indice $i_0<i_1$.Then all the remaining $a$'s in $x^a_{k}$ are matched in $W^{\textsc{fcfm}}(x^a_{k},w)$ 
exactly as the a's in $W^{\textsc{fcfm}}(x^a(k+1),w)$, and so the indices are precisely $i_0,i_1,...,i_l$.
\item Or the first $a$ of $x^a_{k}$ is not matched in $W^{\textsc{fcfm}}(x^a_{k},w)$. Then, all the remaining $a$'s of 
$x^a(k)$ are matched in $W^{\textsc{fcfm}}(x^a_{k},w)$ exactly as the $a$'s in $W^{\textsc{fcfm}}(x^a(k+1),w)$, and the indices are again $i_1,...,i_l$.
\item Or, the first matched $a$ of $x^a_{k}$ in $W^{\textsc{fcfm}}(x^a_{k},w)$ is matched with the letter of index 
$i_1$ in $w$. Then, in $W^{\textsc{fcfm}}(x^a_{k},w)$, either the indices of the matched letters of $w$ 
are the same as in $W^{\textsc{fcfm}}(x^a(k+1),w)$ (and then the last $a$ in $x^a(k)$ remains unmatched), 
or the first $p-k-1$ $a$'s of $x^a_{k}$ are matched with letters of $w$ at indices $i_1,...,i_l$, 
and the last $a$ is matched with a letter of $w$ of index $i_{l+1}$, with $i_l<i_{l+1}$, 
in which case the indices are $i_1,...,i_{l+1}$.
\end{enumerate}
If the indices are $i_1,...,i_l$, then $W^{\textsc{fcfm}}(x^a(k),w)=W^{\textsc{fcfm}}(x^a(k+1),w)$.
If the indices are $i_0,i_1,...,i_l$ or $i_1,...,i_{l+1}$ then there is a letter $b$ of $w$ that is not matched 
with an $a$ of $x^a(k+1)$ in $W^{\textsc{fcfm}}(x^a(k+1),w)$, but is matched with an $a$ of $x^a_{k}$ in $W^{\textsc{fcfm}}(x^a_{k},w)$. Then, either that letter $b$ remains unmatched in $W^{\textsc{fcfm}}(x^a(k+1),w)$, in which case $W^{\textsc{fcfm}}(x^a(k+1),w)$ and $W^{\textsc{fcfm}}(x^a(k),w)$ differ only at index $i_0$ (or $i_{l+1}$), where there is a $b$ in $W^{\textsc{fcfm}}(x^a(k+1),w)$ and $0$ in $W^{\textsc{fcfm}}(x^a(k),w)$. Or, $b$ is matched with a letter $c$ of $w$ in $W^{\textsc{fcfm}}(x^a(k+1),w)$. Then, either the letter $c$ remains unmatched in $W^{\textsc{fcfm}}(x^a(k),w)$, in which 
case $W^{\textsc{fcfm}}(x^a(k+1),w)$ and $W^{\textsc{fcfm}}(x^a(k),w)$ differ only at the place 
of that letter $c$ in $W^{\textsc{fcfm}}(x^a(k),w)$, where there is a $0$ in $W^{\textsc{fcfm}}(x^a(k+1),w)$. 
Or, $c$ is matched with another letter $b'$ in $W^{\textsc{fcfm}}(x^a(k),w)$, in which case we can repeat the same procedure for $b'$ instead of $b$. As we have a finite number of letters in $w$, we eventually stop  with a letter being present in a buffer and $0$ in the other. 
In all cases, the buffers 
$W^{\textsc{fcfm}}(x^a(k+1),w)$ and $W^{\textsc{fcfm}}(x^a(k),w)$ differ only by one letter.
\end{proof}
\noindent For all $\tilde x\in\XX$, let us denote $$T(\tilde x,a) = \mbox{Card}\left\{ \mbox{letters }\tilde x_i \mbox{ of } \tilde x\,:\, \tilde x_i - a \right\}.$$ It follows from the above that 
\begin{corollaire}
\label{buffer}
Let $w$ be a word of length $2p$ such that for some $i \in \llbracket 1;p \rrbracket$ and  $j \in \llbracket p+1; p+i \rrbracket,$ we have  $w_i - w_j$. For such couple $\{i,j\}$, and $k\in \llbracket 0 , 2p \rrbracket$, 
let 
\[x^{i,j}(k)= \begin{cases}
x^{w_i}(k) &\mbox{ if }k \in \llbracket 0,p-1 \rrbracket,\\
\mathbf 0_{2p}&\mbox{ if }k =p,\\ 
x^{w_j}(2p-k)&\mbox{ if }k \in \llbracket p+1,2p \rrbracket.
\end{cases}\] 
Let also $u(k) = W^{\textsc{fcfm}}(x^{i,j}(k),w_1\,\cdots\, w_p),$ for all $k \in  \llbracket 0 , 2p \rrbracket $.
Then, there exists an integer $k$ in  $\llbracket 0 , 2p-1 \rrbracket$, such that 
$u(k) = z_1\,\cdots\, z_p$ differs from $u(k+1) = z'_1\,\cdots\, z'_p$ by only one letter in some position 
$l$, such that $z_l - w_{j}$, $z'_l = 0$,   and for all $h \in \llbracket 1 , p \rrbracket$, $z'_h= w_h$ or $z'_h = 0$. Moreover we have that $T(u(k),w_{j}) = 1$ and $T(u(k+1),w_{j}) = 0$.
\end{corollaire}

\begin{proof}[Proof]
By Lemma \ref{one letter}, for all  $k \in \llbracket 0 , 2p-1 \rrbracket$, $u(k)$ and $u(k+1)$ 
differ at most by one letter in some position $i$ (one being $w_i$, the other being a $0$). 
Therefore, for all $k \in \llbracket 0 , 2p-1 \rrbracket$, $\vert T(u(k),w_{j})- T(u(k+1),w_{j}) \vert \leq 1$. 
Now notice that $2 \leq T(u(0),w_{j})$, because gathering the words $x^{i,j}(0)$ and $w$ would lead to at least 
$p+1$ $w_i$'s out of $2p$ letters - so at least two $w_i$ must remain in $u(0)$. 
On the other hand, we also have that $ T(u(2p),w_{j}) =0 $, because any letter of $w_1\,\cdots\, w_p$ that can 
be matched with $w_{j}$ get matched in $u(2p)$ with the letters of $x^{i,j}(2p)$. 
As a consequence, there exists a rank $k \in \llbracket 0 , 2p-1 \rrbracket $ such that $T(u(k),w_j) =1$ and $T(u(k+1),w_j) =0$. The remaining statements follow readily from Lemma $\ref{one letter}$. 
\end{proof}

\begin{theoreme}
\label{thm:fcfm}
Consider a matching model with impatience $(G,\textsc{fcfm},\mu,p)$.  
Let $w$ be a word of length $2p$ of $\mathbb{V}^*$. Then the following conditions are equivalent:
\begin{enumerate}
\item[(i)]
$w$ is a strongly synchronizing word;
\item[(ii)]
$w$ is a synchronizing word. 
\end{enumerate}
\end{theoreme}
\begin{proof}[Proof]
In view of Theorem \ref{thm:sufficient}, only the implication 
(ii) $\Rightarrow$ (i) remains to be proven. For this,  we reason by contraposition. 
So let $w$ be a word of length $2p$ such that $w_i - w_j$ for some $i \in \llbracket 1;p \rrbracket$ and 
$j \in \llbracket p+1; p+i \rrbracket$. Let $i^* \in  \llbracket 1;p \rrbracket$ and $j^* \in \llbracket p+1; p+i^* \rrbracket$ such that $w_{i^*} - w_{j^*}$ and $j^* = \inf \lbrace j \in \llbracket p+1; 2p \rrbracket$, $ \exists i \in \llbracket 2p -j; p \rrbracket$ $ w_i-w_j  \rbrace$. We denote $u = w_1\cdots w_p$ and $v = w_{p+1}\cdots w_{2p}$.
Let $k^*$ be the integer obtained in Corollary \ref{buffer} for $i\equiv i^*$ and $j\equiv j^*$. 
Then, we get $u(k^*) = d_1\cdots d_p$ and $u(k^*+1) = e_1\cdots e_p$, where $u(.)$ is defined in Corollary  \ref{buffer}. We will prove that $W^{\textsc{fcfm}}(u(k^*),v) \neq W^{\textsc{fcfm}}(u(k^*+1),v)$, 
which will show in turn that $w$ is not a synchronizing word. 

Let $i_1,...,i_l$ be the indices (in increasing order) of letters of $v$ that are matched with letters of $u(k^*)$ 
in $W^{\textsc{fcfm}}(u(k^*),v)$, and $i'_1,...,i'_h$ be the indices (in increasing order) of letters of $v$ that are matched with letters of $u(k^*+1)$ in $W^{\textsc{fcfm}}(u(k^*+1),v)$. Now let us define the following sets, 
\[\begin{cases}
I_0 &=\emptyset;\\
I_{m+1} &= I_m \cup \left\{ \inf \left\{ j \in  \llbracket p+1, p+m+1 \rrbracket  \setminus I_m\,:\,  w_j -  d_{m+1} \right\} \right\},\,m\in \llbracket 0,p-1 \rrbracket. 
 \end{cases}\]
At each step of this construction we add to the set $I_m$ the index of the letter that is matched with  $d_{m+1}$ in $W^{\textsc{fcfm}}(u(k^*),v)$, if any, as in FCFM, 
$d_{m+1}$ is matched with the first compatible letter that has not been matched 
to a previous letter of $u_k^*$. In particular, we finally obtain that $I_p = \lbrace i_1,...,i_l \rbrace$. 
In the same way, we define the sets 
\[\begin{cases}
I'_0 &=\emptyset;\\
I'_{m+1} &= I_m \cup \left\{ \inf \left\{ j \in  \llbracket p+1, p+m+1 \rrbracket  \setminus I'_m\,:\,  w_j -  e_{m+1} \right\} \right\},\,m\in \llbracket 0,p-1 \rrbracket, 
 \end{cases}\]
and the same argument leads to $I'_p = \lbrace i'_1,...,i'_h \rbrace$.

If in Corollary \ref{buffer}, the letter $a$ that can be matched with $w_{j^*}$ in $u(k^*)$ 
(and be replaced by a $0$ in $u(k^*+1)$) is at position $m$, then by construction of $j^*$, $u(k^*)$ and $u(k^*+1)$,  $a$ will indeed be matched with $w_{j^*}$ in $u(k^*)$. 
So the $m$-th step is different for $W^{\textsc{fcfm}}(u(k^*),v)$ and $W^{\textsc{fcfm}}(u(k^*+1),v)$. 
For every other step $m'$, as $d_{m'} \nleftrightline w_{j^*}$ and $e^{m'}\nleftrightline w_{j^*}$, 
we add the same letter, if any, to $I_{m'-1}$ and $I'_{m'-1}$.
So we have $I_p = I'_p \cup \lbrace j^* \rbrace$. 
Let $n_1$ (resp., $n_2$) be the total number of letters from $v$ that are matched 
in $W^{\textsc{fcfm}}(u(k^*),v)$ (resp., $W^{\textsc{fcfm}}(u(k^*+1),v)$). 
As the total numbers of matched letters are even, both
$n_1 + \vert I_p \vert$ and $n_2 + \vert I'_p \vert$ are even. 
But as $\vert I_p \vert$ and $ \vert I'_p \vert $ are of different parity,  so are 
$n_1$ and $n_2$. Thus, 
\begin{align*}
W^{\textsc{fcfm}}(x^{i^*,j^*}(k^*),w) = W^{\textsc{fcfm}}(u(k^*),v) 
&\neq W^{\textsc{fcfm}}(u(k^*+1),v)\\
&= W^{\textsc{fcfm}}(x^{i^*,j^*}(k^*+1),w),\end{align*} 
and $w$ is not a synchronizing word.
\end{proof}

We have proven that being strongly synchronizing is a necessary and sufficient condition for being a synchronizing word of length $2p$ in the case where the matching policy is  {\sc fcfm}. 
It is not the case for all matching policies. 
For example, for the matching policy {\sc lcfm} (Last Come, First Matched'), it can be proven that there exists synchronizing words of length $\lfloor\frac{3p}{2}\rfloor$, so that any suffix of those words that would not satisfy the $p$-condition would still be a synchronizing word. 
However, as we prove hereafter, 
checking that a word is strongly synchronizing is a simple criterion, that can be used to construct an efficient 
perfect sampling algorithm.  



\subsubsection{A second perfect sampling algorithm}
\label{subsubsec:secondalgo}
We are now in position to introduce a perfect sampling algorithm for the state of a matching model with deterministic patience.  

\begin{definition}
\label{defi:Ytilde}
Consider a model $(G,\Phi,\mu,p)$, 
and define for all $k\in\Z$,  the SRS $\tilde Y:=\left(\tilde Y^k_n\right)_{n\ge k}$ on the set  
$\tilde {\mathbb Y} = \lbrace\emptyset \rbrace \cup \bigcup_{j=1}^{2p} \mathbb{V}^{j},$ 
as follows: 
\[\begin{cases}
\tilde Y^k_k &=\emptyset;\\
\tilde Y^k_{n+1} &=\tilde g(\tilde Y^k_{n},V_{n+1})\\
&:=\begin{cases}
v_1\cdots v_iV_{n+1},\mbox{ if $\tilde Y^k_n=v_1\cdots v_i \in \mathbb{V}^i$ with $i<2p$}\\
v_2\cdots v_{2p}V_{n+1},\mbox{ if $\tilde Y^k_n=v_1\cdots v_{2p} \in \mathbb{V}^{2p}$}
\end{cases},
\,n\ge k,
\end{cases}
\]
in a way that for all $k$ and all $n\ge k+2p$, $\tilde Y^k_n$ represents the last $2p$ arrivals to the system at time $n$.
\end{definition}

Consider Algorithm \ref{algo3}.  
It consists of another declination of Algorithm $\ref{algo1}$,  
started with $\tilde Y = \emptyset$, for $\tilde Y$ the recursion of Definition \ref{defi:Ytilde}, $b_1,...,b_q$, the strongly synchronizing words of the model, and $a_1,...,a_q$, the states of $\tilde X$ after the arrival of $b_1$,...,$b_q$, respectively. 
\begin{algorithm}
\caption{Simulation of the stationary probability of $\tilde X$ - Matching model with deterministic patience}
\KwData {A probability distribution $\mu$ on $\mathbb{V}$} 
 $T_{down} \gets -1$ \;
 $T_{up} \gets -2p$ \tcc*[l]{We initialize the starting time at time $2p$}
 $\tilde Y \gets \emptyset$ \;
\While{$\tilde Y \,\mbox{{\em is not strongly synchronizing}}$}
{ $i \gets T_{up}$ \;
 $Y \gets \emptyset$ \;
\For{$j \gets T_{up} \ \KwTo \ T_{down} $}
{{ \bfseries{draw} $v_j$  \bfseries{from} $\mu$\tcc*[l]{We draw the input at this iteration}}}
\While{$i<0$ {\em {\bf and}} $\tilde Y \,\mbox{{\em is not strongly synchronizing}}$}
{$\tilde Y \longleftarrow g(\tilde Y,v_i)$\tcc*[l]{We investigate all arrival scenarios of length $2p$ from $T_{up}$ to
0, and stop if one of them is strongly synchronizing.}
$i \longleftarrow i+1$ \;}
 $T_{down} \gets T_{up}-1$ \;
$T_{up} \gets 2T_{up}$ \;
}

$\tilde X \longleftarrow z(\tilde Y) $ \tcc*[l]{We assign to $\tilde X$ the common state 
induced by the synchronizing word 
$\tilde Y$}
\tcc{ We now transition $\tilde X$ to time $0$ as a matching system.}
\While{$i<0$}
{$\tilde X \longleftarrow \tilde f^\Phi(\tilde X,v_i)$ \;
$i \longleftarrow i+1$ \;}
\KwRet{$\tilde X$} 
\label{algo3}
\end{algorithm}
\noindent We have the following result,  
\begin{proposition}
$\tilde X$ is recurrent positive. Moreover, Algorithm \ref{algo3}  terminates a.s., and its output is sampled from the stationary distribution of $\tilde X$.
\end{proposition}
\begin{proof}[Proof]
We can easily show that $\tilde Y$ $q$-controls $\tilde X$, with $q$ the number of strongly synchronizing word. 
Let $w$ be a strongly synchronizing word. By Theorem \ref{thm:sufficient}, 
$w$ is a synchronizing word. Thus for all $k\in\Z$ and $n\ge k$, we get in particular that 
\begin{equation}
\left[ \tilde Y^{k}_n(\emptyset)= w \right] \Longrightarrow \left[\forall \tilde x \in \XX,\, \tilde X^{k}_n(\tilde x)= W^{\Phi}(\emptyset,w) \right],
\end{equation}
which implies that $\tilde Y$ controls $\tilde X$ over all strongly synchronizing words. 
We conclude using Theorem \ref{thm:main}.
\end{proof}
\begin{remark}
Observe that $\tilde Y$ is not irreducible, however it reaches its recurrent class in $2p$ iterations. So for all strongly synchronizing word $w$, we still have that $$\Pb(\tau_\emptyset^{\tilde Y}(w) < \infty) = 1.$$
\end{remark}

\subsubsection{Efficiency of Algorithm \ref{algo3}}
\label{subsubsec:effi}

In this section we analyse the coalescence time of Algorithm \ref{algo3}. For this, one needs to assess the probability that a given input word of length $2p$ is strongly synchronizing. This is, in turn, a function of $\mu$ and of the number of admissible arrival words of length $2p$ that are strongly synchronizing. 
The latter number is, clearly, highly dependent on the geometry of the compatibility graph at hand. 

Let us first bound the average number of iterations of the algorithm to see the coalescence time, and then for the corresponding horizon in the past, in function of the number of strongly synchronizing words. We have the following, 

\begin{proposition}
\label{prop:saul}
Let $I$ be the number of iterations of Algorithm \ref{algo1} to detect coalescence, and $T=-p2^I$ be the corresponding starting time. 
Then, we have that 
$$\esp{-T} \le {2p \over q^{p,\mu}}\,,$$
where 
\[q^{p,\mu}=\pr{V_1\cdots V_{2p}\mbox{ is strongly synchronizing}}.\] 
\end{proposition}
\begin{proof}[Proof]
For any integer $n\ge 1$, we let for all $i\in\N^*$, $z^n_i $ be the word of length $2p$ representing the 
arrivals into the system between time $-p2^n+(i-1)2p$ and time $-p2^n+i2p-1$ included, in the order of arrivals. 
We also let 
\[K^n=\inf\left\{i \in \N^*\,:\,z^n_i\mbox{ is strongly synchronizing}\right\}.\]
The independence of arrivals implies that the r.v.'s $K^n,\,n\in\N^*$ are identically distributed (but not independent) of 
geometric distribution of parameter 
\[q^{p,\mu}=\pr{V_1\cdots V_{2p}\mbox{ is strongly synchronizing}}.\]
Now, it readily follows from Theorem \ref{thm:sufficient}, that for all $n\in\N^*$, 
$I \le n$ in particular if there has been a strongly synchronizing arrival array between 
times $-p2^{n}$ and $-1$ included, that is, if $2pK^n \le p2^n$. Consequently, for all $n\in\N^*$ we get that 
\begin{equation*}
\pr{-T > p2^n} = \pr{I >n} \le \pr{2pK^n > p2^n}=\pr{2p K^1 > p2^n}.
\end{equation*}
This readily implies that $-T \le_{\tiny{\mbox{st}}} 2p K^1$, where $\le_{\tiny{\mbox{st}}}$ denotes the strong stochastic ordering. 
We deduce that $$\esp{-T} \le 2p\esp{K^1}={2p \over q^{p,\mu}}.$$

\end{proof}
Whenever the arrival measure $\mu$ is uniform over $\mathbb{V}$, the latter results specializes as follows, 
\begin{corollaire}
\label{cor:borneNunif}
If the graph $G=(\mathbb{V},E)$ is of size $n$ and $\mu$ is uniform over $\mathbb{V}$, we get the bounds 
$$\esp{-T} \le {2p n^{2p}\over N(G,p)},\quad\esp{I} \le 1 + {2p\mbox{{\em Log}} n-\mbox{{\em Log}} N(G,p) \over \mbox{{\em Log}} 2},$$
where $N(G,p)$ is the number of strongly synchronizing words of $\mathbb{V}^*$. 
\end{corollaire}

\begin{proof}[Proof.]
The results readily follow from Proposition \ref{prop:saul}, observing that in this case 
$$q^{p,\mu}={N(G,p)\over n^{2p}}\cdot$$
\end{proof}

For a given $G$ and a given $p$, computing the number $N(G,p)$ of strongly synchronizing words, is of crucial interest to assess the efficiency of Algorithm \ref{algo1}. As Corollary \ref{cor:borneNunif} demonstrates, a function of the latter quantity provides bounds for the expected values of $|T|$ and $I$. We now turn to a specific evaluation of $N(G,p)$, and for this, we first need the following definitions, 

\begin{definition}
Let $(G = (\mathbb{V},E),\Phi,\mu,p)$ be a discrete matching model with impatience. For any strongly synchronizing word $w=w_1\cdots w_{2p}$, the {\em trace} of $w$ is defined as the word $Z^w$ gathering, in their order of apparences, all distinct letters 
of the second half of $w$. In other words, we set 
\begin{enumerate}
\item [(1)]
$Z^w_1 = w_{p+1}, $
\item [(2)]
For all $i \in \llbracket 1,p-1\rrbracket$,
\[Z^w_{i+1} = \begin{cases} 
Z^w_i,&\,\mbox{ if }w_{p+i+1} \in Z^w_i;\\
Z^w_i \ w_{p+i+1},&\,\mbox{  if }w_{p+i+1}\notin Z^w_i,
\end{cases}\]
\end{enumerate}
and $Z^w \equiv Z^w_p$. 
\end{definition}
In what follows, for any word $z=z_1\,\cdots\, z_l$, we denote by $\beta(z)$ the cardinality of the set of nodes that are incompatible with all letters of 
$z$, namely 
$$\beta(z) = \mbox{Card} \Bigl\{ v \in \mathbb{V}\,:\, \forall i \in \llbracket 1,l \rrbracket,  v \nleftrightline z_i \Bigl\}.$$

\noindent We have the following, 
\begin{proposition}
\label{prop:count1}
Let $(G = (\mathbb{V},E),\Phi,\mu,p)$ be a discrete matching model with impatience, and let $\mathscr T(G)$ be the set of words having distinct letters, 
that form a permutation of the elements of a set $U\subset \mathbb{V}$ that is such that $E(U) \ne \mathbb{V}$. 
Then, the number $N(G,p)$ of strongly synchronizing words is given by 
\begin{equation*}
N(G,p)
= \sum\limits_{z= z_1\cdots z_l \in \mathscr T(G)} \sum\limits_{\lbrace 1=k_1<k_2<\cdots <k_{l}< k_{l+1}= p+1 \rbrace} \prod_{i=1}^{l} i^{k_{i+1}-k_{i}-1}\beta(z_1 z_2\cdots z_i)^{k_{i+1}-k_{i}}. 
\end{equation*}
\end{proposition}
\begin{proof}
Let $z=z_1\,\cdots \,z_l\in \mathbb{V}^*$ be a word having $l$ distinct letters. 
Let us denote, for any word $w = w_1\,\cdots \,w_{2p}$ of trace $z$, 
\[k^w_i  = \inf \Bigl\lbrace j \in \llbracket 1, p\rrbracket,\,w_{p+j} = z_{i} \Bigl\},\quad i \in \llbracket 1, l \rrbracket,\]
the consecutive indexes, in the second half suffix of $w$, corresponding to the first occurrences of the successive letters of $z$. 

Let $1=k_1<k_2<\cdots <k_{l} <k_{l+1}=p+1$ be a fixed family of integers and $w$ be a word of length $2p$. We first show the equivalence between the two following assertions, 
\begin{enumerate}
    \item [(i)]
    $w$ is strongly synchronizing, has trace $z = z_1 \cdots z_l$, and 
    $$(k^w_1,k^w_2,...,k^w_l) =(k_1,k_2,...,k_{l}).$$
    \item[(ii)]
    For all $i \in \llbracket 1 , l \rrbracket$,\\
    \begin{itemize}
    	\item[(iia)] $w_{p+k_i} =  z_{i}$;\\
	\item[(iib)] \mbox{For all } $j \in \llbracket k_{i},k_{i+1} -1 \rrbracket,\, w_{p+j} \in \lbrace z_1 , \cdots, z_i \rbrace$ and $w_{j} \in E\left(\{z_1,\,\cdots\,,z_i\}\right)^c.$
    \end{itemize}
\end{enumerate}
\medskip 

Indeed, if (i) holds true then (ii) also holds by induction on $i$: First, (iia)-(iib) hold true for $i=1$. Indeed, 
for all $j \in \llbracket k_{1},k_{2} -1 \rrbracket$ 
we have that $w_{p+j} =z_1$ by definition of $k_1$ and $k_2$, and thus, by definition of a strongly synchronizing word, that $w_{j}  \nleftrightline w_{p+k_1}=z_{1}$.  
Now suppose that (iia)-(iib) hold true for some 
$i-1\in\llbracket 1,p-1 \rrbracket$. Then (iia) holds for $i$ by definition of the trace and of $k_i$. (iib) also holds true by induction on $j$ over $\llbracket k_{i},k_{i+1} -1\rrbracket$: First, we have $w_{p+k_i} =  z_{i}$ by (iia), implying, by definition of a strongly synchronizing word and  in view of the induction assumption, that 
\begin{align*}
w_{k_i} &\in E\left(\left\{w_\ell\,:\,\ell \in \llbracket p+1,p+k_i \rrbracket\right\}\right)^c\\
&\in E\left(\left\{w_\ell\,:\,\ell \in \llbracket p+1,p+k_i-1 \rrbracket\right\}\cup\{w_{p+k_i}\}\right)^c\\
&=E\left(\left\{z_1,\cdots,z_i-1\right\}\cup\{z_i\}\right)^c=E\left(\left\{z_1,\cdots,z_i\right\}\right)^c,
\end{align*}
so the properties in (iib) hold for $j=k_i$. Now suppose that they hold true for some $j-1\in \llbracket k_{i},k_{i+1} -2\rrbracket$. Then 
$w_{p+j} \in \{z_1,\cdots,z_i\}$ by the very definition of $k_i$. 
Thus, as $w$ is strongly synchronizing we have that 
\begin{align*}
w_{j} &\in 
E\left(\left\{w_\ell\,:\,\ell \in \llbracket p+1,p+j \rrbracket\right\}\right)^c\\
&=E\left(\left\{w_\ell\,:\,\ell \in \llbracket p+1,p+j-1 \rrbracket\right\}\cup \{w_{p+j}\}\right)^c\\
&=E\left(\left\{z_1,\cdots,z_i\right\}\cup  \{w_{p+j}\}\right)^c
=E\left(\left\{z_1,\cdots,z_i\right\}\right)^c.
\end{align*}
Thus (iib) hold true at index $i$, which completes the proof of (ii). 

Now suppose that (ii) holds. Then it follows from (iia) and the first property in (iib) that $k^w_i=k_i$ for all 
$i\in\llbracket 1,l \rrbracket$. Now fix $j\in\llbracket 1,p \rrbracket$, and let $i$ be the index in 
$\llbracket 1,l \rrbracket$ such that $j\in \llbracket k_i,k_{i+1}-1 \rrbracket$. 
Then, in view of (iia)-(iib) we get that 
\[w_{j} \in E\left(\{z_1,\,\cdots\,,z_i\}\right)^c=E\left(\{w_{p+1},\,\cdots\,,z_{p+j}\}\right)^c,\]
so $w$ is indeed strongly synchronizing. From (ii), $w$ also clearly has trace $z$, so (i) holds, which concludes the proof of (i) $\Leftrightarrow$ (ii).

%

Now, for a fixed trace $z$, to count the strongly erasing words having trace $z$ it thus suffices to count, for all families of integers $1=k_1<k_2<\cdots<k_l<k_{l+1}=p+1$, all the words $w$ satisfying (ii). First, the letters at indices $k_1$,...,$k_{l}$ are fixed, and for all $i \in \llbracket 1, l \rrbracket$ we have $i$ possibilities for each letter between indices $k_i+1$ and $k_{i+1}-1$, and $\beta(z_1\cdots z_i)$ possibilities for each letter between indices $k_{i}-p$ and $k_{i+1}-1-p$. 
Therefore, 
the number of strongly synchronizing words and having $z$ has a trace is given by 
\begin{equation}\label{eq:defNz}
N_z:=
\sum\limits_{\lbrace 1=k_1<k_2<\cdots <k_{l}< k_{l+1}= p+1 \rbrace} \prod_{i=1}^{l} i^{k_{i+1}-k_{i}-1}\beta(z_1 z_2\cdots z_i)^{k_{i+1}-k_{i}}. 
\end{equation}
Last, to get $N(G,p)$ we must sum the above quantity over all possible traces of strongly synchronizing words. 
To characterize this set, observe that any trace $z$ necessarily has distinct letters, forming a permutation of a set $U_z\subset \mathbb{V}$. If $E(U_z) \ne \mathbb{V}$, then there exists a letter $i \in  \mathbb{V}\setminus E(U_z)$, and $z$ clearly is the trace of any word $w$ whose prefix of size $p$ is $ii\,\cdots \,i$, and whose suffix of size $p$ is a permutation of the elements of $U$. If now $E(U_z) =\mathbb{V}$, 
as for any strongly synchronizing word $w$ and having trace $z$ we must have that $w_p \not\in E(U_z)$, leading to an immediate contradiction. Thus it is necessarily and sufficient that $E(U_z) \ne \mathbb{V}$ for $z$ to be a trace, which concludes the proof. 
\end{proof}

%

\subsubsection{Example}
To illustrate the efficiency of Algorithm \ref{algo3} in the case of deterministic patience times, we consider the 
simple non-trivial example of the so-called {\em paw} graph $G$ of Figure \ref{fig:paw}. 

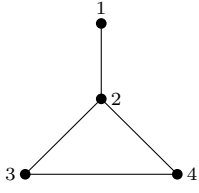
\begin{figure}[htb]
	\begin{center}
		\begin{tikzpicture}
		\fill (0,2) circle (2pt)node[above]{\scriptsize{1}};
		\fill (0,1) circle (2pt)node[right]{\scriptsize{2}};
		\fill (-1,0) circle (2pt)node[left]{\scriptsize{3}};
		\fill (1,0) circle (2pt)node[right]{\scriptsize{4}};		
		\draw[-] (0,1) -- (1,0);
		\draw[-] (0,1) -- (-1,0);
		\draw[-] (0,1) -- (0,2);
		\draw[-] (-1,0) -- (1,0);		
		\end{tikzpicture}
		\caption{The paw graph.}
		\label{fig:paw}
	\end{center}
\end{figure}

As the above results demonstrate, for any $p$, to compute the number $N(G,p)$ of strongly synchronizing words we first need to determine the set of all possible traces $\mathscr T(G)$ of $G$. In the present case we readily obtain that  
\[\mathscr T(G)=\Bigl\{1,2,3,4,13,14,31,34,41,43,134,143,314,341,413,431\}.\]

Indeed, any trace containing a $2$ can only contain $2$ since adding another class will result in having compatible classes which can't be the case for a trace. Conversely, Any other word containing $1$ or $3$ or $4$ is possible as a trace since not having $2$ in the word means that the letters of the word are still compatible with the class $1$. 
It is then immediate to compute $\beta(z)$ for all $z\in\mathscr T(G)$ using (\ref{eq:defNz}). We obtain 
\begin{equation}
\label{eq:Nz}
\begin{cases}
N_{\footnotesize{1}}=3^p,\,N_2=1,\,N_3=2^p,\,N_4=2^p,
N_{13} =\frac{3}{2} 4^p - 2.3^p,\,N_{14}=\frac{3}{2} 4^p - 2.3^p,\\
N_{31}=\frac{1}{2}4^p - 2^p,\,N_{34}=(p-1) 2^{p-1},\,
N_{41}=\frac{1}{2}4^{p} - 2^p,\,N_{43}= (p-1) 2^{p-1},\\
N_{134}= 3.2^{2p-1} - (2p+4)3^{p-1},\,N_{143} = 3.2^{2p-1} - (2p+4)3^{p-1},\\
N_{314}= 2^{2p-1} - 4.3^{p-1}+2^{p},\,N_{341}= 2.3^{p-1} -(p+1) 2^{p-1},\\
N_{413}= 2^{2p-1} - 4.3^{p-1}+2^{p},\,N_{431} = 2.3^{p-1} -(p+1) 2^{p-1}.
\end{cases}\end{equation}
For clarity, let us detail one of the above computations, for $z = 13$. We then have 
$\beta(1) = \vert \lbrace 1,3,4 \rbrace \vert =  3$ and $\beta(13) = \vert  \lbrace 1,3 \rbrace\vert = 2$. Therefore, using (\ref{eq:defNz})  we have

\begin{align*}
N_{13} &=  \sum_{k_2=2}^{p} 1^{k_2-1-1}\beta(1)^{k_2-1} 2^{p+1-k_2-1}\beta(13) ^{p+1-k_2}\\
&= \sum_{k_2=2}^{p} 3^{k_2-1} 2^{p-k_2} 2^{p+1-k_2}\\
&= \frac{2\times 4^p}{3} \sum_{k_2=2}^{p} \left(\frac{3}{4}\right)^{k_2}
= \frac{3}{2} 4^p - 2.3^p. 
\end{align*}
Summing all elements of \eqref{eq:Nz} and rearranging, we obtain that  
\begin{equation*}
N(G,p) = 1 +2^{2p+3} - 3^{p+1} -4(p+3)3^{p-1}.
\end{equation*}
Then, applying Corollary \ref{cor:borneNunif} and Jensen's inequality we obtain the following bound for the average number of iterations of Algorithm \ref{algo1} to detect coalescence, 
\begin{align*}
\esp{I} &\le 1 + {2p\mbox{{ Log}} n-\mbox{{Log}} N(G,p) \over \mbox{{Log}} 2}\\
	   &=1+4p-{\mbox{{Log}} \left(1 +2^{2p+3} - 3^{p+1} -4(p+3)3^{p-1}\right) \over \mbox{{Log}} 2}=:B_I,
\end{align*}
and the average starting time $T$ to detect coalescence is bounded  by 
$-p2^{B_I}$. In Table 1, we specify the number of strongly synchronizing words, together with the corresponding bounds for $\esp{I}$ and 
$\esp{-T}$, for various values of $p$. 

\begin{table}
\begin{tabular}{|c|c|c|c|}
    \hline
    $p$ & $N(G,p)$ & Bound for $\esp{I}$ & Bound for $\esp{-T}$\\
    \hline
    1 & 8 & 2 & 4\\
     \hline
    2 & 42 & 3,608 & 24,381\\
     \hline
    3	& 216	& 5,245 &	113,778\\
    \hline
4 &1050	& 6,964 &	499,322\\
\hline
5 &	4872	 &8,750 &	2152,250 \\
\hline
6 &	21834 &	10,586 &	9220,784\\
\hline
7 &	95352 &	12,460 &	39412,874\\
\hline
8 &	408378 &	14,360 & 168274,189\\
\hline
9 &	1723176 &	16,283 &	717831,830 \\
\hline
10 &	7187946 &	18,223 &	3059320,779\\
\hline
\end{tabular}
\label{table:paw}
\caption{Efficiency of Algorithm \ref{algo1}.}
\end{table}



\subsubsection{Complexity comparison}
\label{subsubsec:compare}
After having provided a bound for the average coalescence time for Algorithm \ref{algo3}, we now compare the number of operations necessary to complete Algorithm \ref{algo3}, to the number of operations necessary to complete the primitive CFTP algorithm, consisting of running chains started from all possible states, in parallel. 
To compare those two algorithms, we need to specify what we mean by {\em operations}: We say that an algorithm does one operation if it compares two letters of $\mathbb{V}$, to determine if they are equal or not or if the two letters are connected in $G$. It is intuitively clear, that 
the two algorithms can be basically decomposed into a sequence of such operations: 
\begin{itemize}
\item In the CFTP algorithm, the match of the incoming individuals amounts to an investigation of the set of stored compatible items in a determinate order, and thereby, of a sequence of such operations. Second, so does the test of equality of the current states of all Markov chains, at any given time.
\item In Algorithm \ref{algo3}, testing the `strong synchronizing' property at all time is again a sequence of operations, and so does the construction of the dynamics of the recursion, from the coalescence time on. 
\end{itemize} 
To estimate the number of operations in the two algorithms, for two values of $p$ (3 and 6), we have 
first drawn realizations of Erdös-Rényi graphs $G$ of parameters $(n,q)$, that are conditioned to be connected, for various values of the size $n$ and of the connectivity parameter $q$. We have then tracked the average number of operations for 10 realizations of both algorithms, on the same graph each time. 
The results are presented in Table 2 and 3.  

\begin{table}
\begin{tabular}{|c|c|c|c|c|c|c|}
    \hline
    $p=3$& $q=\frac{1}{8}$& $q=\frac{2}{8}$& $q=\frac{3}{8}$& $q=\frac{4}{8}$& $q=\frac{5}{8}$ 
    \\
    \hline
    \hline
    $n = 4$, CFTP & $   2907.67$&$3356.04$ &$3012.46$ &$3228.72$ &$3393.52$ 
    \\
    \hline
    $n=4$, Algo \ref{algo3}&$123.16$&$  168.11$&$ 297.64$&$ 213.27$&$ 307.5$ 
    \\
    \hline
    \hline
    $n=5$, CFTP
    &$4689.57$&$ 5078.86$&$ 4694.66$&$ 5542.87$&$ 4713.41$ 
        \\
    \hline
    $n=5$, Algo \ref{algo3}
      &$102.2$ &$108.24$ & $161.58$ & $422.99$&$548.35$ 
    \\
    \hline
    \hline
    $n=6$, CFTP
      &$7888.29$&$ 7350.42$& $6550.72$ &$7466.3$ & $7319.94$  
    \\
    \hline
    $n=6$, Algo \ref{algo3}
    &$93.96$ & $82.39$  & $161.96$ & $338.36$ & $406.9$ 
    \\
    \hline
    \hline
    $n=7$, CFTP
    &$11458.40$&$10200.46$ &$111044.26$ &$10455.48$ &$9222.91$ 
    \\
    \hline
    $n=7$, Algo \ref{algo3}
    & $69.06$ &$117.94$&$140.82$& $252.71$ & $764.8$ 
    \\
    \hline
    \hline
    $n=8$, CFTP
    &$15984.74$ &$14829.7$ &$15127.1$ &$12565.42$ &$12189.06$ 
    \\
    \hline
    $n=8$, Algo \ref{algo3}
    &$56.85$ &$86.28$ & $93.19$ &$241.3$ & $818.16$ 
    \\
    \hline
\end{tabular}
\captionof{table}{Average number of operations of the algorithms for $10$ repetitions with $p=3$ and multiple values of $(n,q)$.}
\end{table}

\begin{table}
\begin{tabular}{|c|c|c|c|c|c|c|}
    \hline
    $p=3$& $q=\frac{1}{8}$& $q=\frac{2}{8}$& $q=\frac{3}{8}$& $q=\frac{4}{8}$& $q=\frac{5}{8}$ 
    \\
    \hline
    \hline
    $n = 4$, CFTP & $0.00969$&$0.01$  &$0.00890$ &$0.00937$ &$ 0.01$
    \\
    \hline

    $n=4$, Algo \ref{algo3}&$0.00078$&$ 0.00093$&$  0.00172$&$ 0.00125$&$0.00172$ 
    \\
    \hline
    \hline

    $n=5$, CFTP
    &$0.01282$&$ 0.01468$&$ 0.01328$&$ 0.01609$&$  0.01422$ 
        \\
    \hline
    $n=5$, Algo \ref{algo3}
      &$0.00047$ &$0.00047$ & $0.00125$ & $0.00280$&$0.00328$ 
    \\
    \hline
    \hline
    $n=6$, CFTP
      &$  0.02236$&$ 0.02046$& $ 0.01875$ &$0.02188$ & $0.02375$  
    \\
    \hline
    $n=6$, Algo \ref{algo3}
    &$0.00063$ & $0.00062$  & $0.00110$ & $0.00220$ & $0.00233$ 
    \\
    \hline
    \hline

    $n=7$, CFTP
    &$0.03499$&$0.03608$ &$0.03281$ &$0.03407$ &$0.02954$ 
    \\
    \hline
  
    $n=7$, Algo \ref{algo3}
    & $0.00046$ &$0.00062$&$ 0.00079$& $0.00187$ & $ 0.00484$ 
    \\
    \hline
    \hline
    $n=8$, CFTP
    &$0.04984$ &$0.04545$ &$ 0.04404$ &$0.03125$ &$0.02922$ 
    \\
    \hline
    
    $n=8$, Algo \ref{algo3}
    &$0.00048$ &$0.00078$ & $0.00078$ &$0.00172$ & $0.00594$ 
    \\
    \hline
\end{tabular}
\captionof{table}{Average cputime of the algorithms on standard computer for $10$ repetitions with $p=3$ and multiple values of $(n,q)$.}
\end{table}

\begin{table}
\begin{tabular}{|c|c|c|c|c|c|c|}
    \hline
    $p=6$& $q=\frac{1}{8}$& $q=\frac{2}{8}$& $q=\frac{3}{8}$& $q=\frac{4}{8}$& $q=\frac{5}{8}$
    \\
    \hline
    \hline
    $n = 4$, CFTP & $467834.38 $&$477424.56 $ &$457725.73 $ &$433542.6 $ &$363723.37 $ 
    \\
    \hline
    $n=4$, Algo \ref{algo3} &$5576.1 $&$12843.9 $ &$12070.8 $&$21559.18 $ & $17672.3$ 
    \\
    \hline
    \hline
    $n=5$, CFTP
    &$1248551.28$ &$1139776.29$ &$919830.29$ &$853625.28$ &$980490.29$ 
        \\
    \hline
    $n=5$, Algo \ref{algo3}
      &$12472.92$ &$11666.23$ &$11111.01$ &$58257.99$ &$143122.55$ 
    \\
    \hline
    \hline
    $n=6$,CFTP

      &$3218753.63$ &$2446069.32$ &$2999130.02$ &$2128500.86$ &$1547150.53$ 
    \\
    \hline
    $n=6$, Algo \ref{algo3}
    &$3183.32$ &$6274.56$ &$11272.88$ &$127429.52$ & $284116.14$ 
    \\
    \hline
    \hline
    $n=7$, CFTP
    &$4790047.9$ &$7288225.3$ &$7117622.9$ &$2536934.7 $&$3628303.2$ 
    \\
    \hline
    $n=7$, Algo \ref{algo3}
    &$2609.97$ &$9818.73$ & $455.36$ & $171196.79$ &$381580.97$  
    \\
    \hline
    \hline
    $n=8$, CFTP
    &$14779764.95$ &$10594880.96$ &$8477686.16$ &$6073463.72$ &$4123539.06$ 
    \\
    \hline
    $n=8$, Algo \ref{algo3}
    &$2382.32$ &$2174.99$ &$14180.58$ & $46050.05$ &$389028.98$ 
    \\
    \hline
\end{tabular}
\captionof{table}{Average number of operations of the algorithms for $10$ repetitions with $p=6$ and multiple values of $(n,q)$.}
\end{table}

\begin{table}
\begin{tabular}{|c|c|c|c|c|c|c|}
    \hline
    $p=3$& $q=\frac{1}{8}$& $q=\frac{2}{8}$& $q=\frac{3}{8}$& $q=\frac{4}{8}$& $q=\frac{5}{8}$ 
    \\
    \hline
    \hline
    $n = 4$, CFTP & $1.10219$&$1.20124$  &$1.13655$ &$1.09686$ &$ 0.91844$
    \\
    \hline
    $n=4$, Algo \ref{algo3}&$0.05930$&$ 0.1313$&$ 0.1328$&$ 0.24062$&$0.2062$ 
    \\
    \hline
    \hline

    $n=5$, CFTP
    &$2.85954$&$ 2.66343$&$ 2.06705$&$ 2.23469$&$  2.66578$ 
        \\
    \hline

    $n=5$, Algo \ref{algo3}
      &$0.06749$ &$0.08233$ & $0.06798$ & $0.55468$&$2.4180$ 
    \\
    \hline
    \hline
    $n=6$, CFTP
      &$ 7.99048$&$ 6.53797$& $ 9.8092$ &$5.83189$ & $4.18533$  
    \\
    \hline
    $n=6$, Algo \ref{algo3}
    &$0.01812$ & $0.03922$  & $0.07391$ & $1.31535$ & $3.8147$ 
    \\
    \hline
    \hline

    $n=7$, CFTP
    &$17.61454$&$14.21235$ &$15.81282$ &$10.9244$ &$17.6595$ 
    \\
    \hline
  
    $n=7$, Algo \ref{algo3}
    & $0.01516$ &$0.06514$&$ 0.03921$& $3.53732$ & $ 7.23748$ 
    \\
    \hline
    \hline
    $n=8$, CFTP
    &$42.85513$ &$35.96576$ &$ 27.95357$ &$18.95672$ &$12.183$ 
    \\
    \hline
    
    $n=8$, Algo \ref{algo3}
    &$0.01517$ &$0.01281$ & $0.09954$ &$46.98482$ & $13.43019$ 
    \\
    \hline
\end{tabular}
\captionof{table}{Average cputime of the algorithms on standard computer for $10$ repetitions with $p=6$ and multiple values of $(n,q)$.}
\end{table}


The results gathered in Tables 2 and 3 tend to indicate that Algorithm \ref{algo3} is much more efficient than primitive CFTP, and that the performance gap is particular important for sparse graphs. This last fact is an intuitively clear consequence of 
the fact that the proportion of strongly synchronizing words is decreasing in the number of edges. For $q \geq \frac{3}{4}$, however, we observe cases where the Algorithm \ref{algo3} does not terminate in a reasonable amount of time.

\subsection{Deterministic matching model with latency}
\label{subsec:latency}
It is well known that the primitive CFTP algorithm is in general not a good benchmark in terms of complexity, as it requires the coalescence of a large number of versions of the considered Markov chain - which makes its use unpractical for a large state space. On the other hand, and as previously mentioned, non-trivial deterministic matching models do not satisfy condition \eqref{eq:condstabmatching}.
This means that the SRS defined by the recursion \eqref{eq:recurY} cannot hit $0$, and so we cannot use Algorithm \ref{algo2} for deterministic matching models. 

In this section we introduce the following variant of the model of Section \ref{subsec:detpatience}: we suppose that {latency} is allowed, that is, at each instant we suppose that, with a positive probability $\gamma$, no item enters the system. 
In other words, the generic inter-arrival time $\xi$ follows a geometric law of parameter $1-\gamma$. 
The sequences $\suitez{\widehat V_n}$ and $\suitez{\widehat P_n}$ 
are then defined as follows: 
\begin{itemize}
    \item If an item enters the system at time $n$, then $\widehat V_n$ is the class of the item entering the system, and $\widehat P_n = p+\varepsilon$ is the patience of the item entering the system;
    \item Else, we set $\widehat V_n = -1$ and $\widehat P_n = 0$.
\end{itemize}
We denote such deterministic model with latency, by $(G,\Phi,\mu,p,\gamma).$ 
Similarly to Section \ref{subsec:detpatience}, we then easily obtain a simplified representation of the system state:
\begin{definition}
For all $n\in\Z$, the {\em word-profile} of the system just before time $n$ is defined by the word 
$$ \widehat X_n = w_1\cdots w_p \in (\mathbb{V}\cup\{0\}\cup\{-1\})^p,$$
 where for all $i \in \llbracket 1,p \rrbracket$, 
\[w_ i = \begin{cases}
         \widehat{V}_{n-p+i-1} = v \in \mathbb{V}  & \mbox{if the item of class v entered at time  }n-p+i-1\\
         &\mbox{ was not matched before }\small{n};\\
          0 & \mbox{if the item entered at time } n-p+i-1\\ 
          &\mbox{ was matched before } \small{n}.\\
          \widehat{V}_{n-p+i-1}= -1 & \mbox{if no item entered at time } n-p+i-1.
          \end{cases}.\]
\end{definition}
We can therefore view the latency at a certain time $n$, as the arrival of an item labeled 
$-1$, that cannot be matched with any other item. We then denote by \[\hX = (\mathbb{V}\cup\{0\}\cup\{-1\})^p,\]
the (finite) state space space of $\widehat{X}$. Contrary to the model of Section \ref{subsec:detpatience} (which can be seen as a particular of the present one for $\gamma=0$), \eqref{eq:condstabmatching} is verified here, since 
the geometric r.v. $\xi$ can be arbitrarily large. 
Therefore, Algorithm \ref{algo2} can be used to design a perfect sampling algorithm in this case. Let $ \widehat{Y}:=\suitez{\widehat{Y}_n}$ be the SRS defined by the recursive equation 
\begin{equation}
\label{eq:recurYhat}
\widehat{Y}_{n+1}=\left[\max(\widehat{Y}_n,\widehat{P}_n)-1\right]^+= g\left(\widehat{Y}_n,\widehat{P}_n\right),\quad n \in \Z. 
\end{equation}
\begin{proposition}
\label{theolat}
For any $n \in \Z$, the following statements are equivalent:
\begin{enumerate}
    \item [i)]
$\widehat{Y}_n = 0$
    \item [ii)]
    For all $k\in \llbracket 1, p\rrbracket,\widehat{V}_{n-k} = -1$ (and equivalently $\widehat{P}_{n-k} = 0$).
\end{enumerate}
\end{proposition}

\begin{proof}
Fix $n\in\Z$. 
Regarding the implication ii) $\Rightarrow$ i), by the construction of $\widehat{P}$ and $\widehat{Y}$ we have that 
\begin{equation*}
\widehat{Y}_{n-p} \leq p+\varepsilon-1.
\end{equation*}
Moreover, for all $k \in \llbracket 1 , p\rrbracket$, 
\begin{equation*}
\widehat{Y}_{n-k+1} = \left[\max(\widehat{Y}_{n-k},{\widehat{P}_{n-k}})-1\right]^+ =  \left[\max(\widehat{Y}_{n-k},0)-1\right]^+ = \left[\widehat{Y}_{n-k}-1\right]^+,
\end{equation*}
so by induction we obtain that 
\begin{equation*}
\widehat{Y}_n \leq \max(\widehat{Y}_{n-p}-p,0) \leq \max(p{+\varepsilon}-1-p,0) = 0.
\end{equation*}

We now turn to the converse implication i) $\Rightarrow$ ii) : Suppose, to the contrary, that for some 
$k \in \llbracket 1 , p \rrbracket$ we have  $\widehat{P}_{n-k} \neq 0$ which, 
by the very definition of $\widehat{P}$, means that $\widehat{P}_{n-k} = p+\varepsilon$.  Then, 
as $\widehat{Y}_{n-k} \leq p+\varepsilon-1$ we have that 
\begin{equation*}
\widehat{Y}_{n-k+1} = \left[\max(\widehat{Y}_{n-k},\widehat{P}_{n-k})-1\right]^+ = \widehat{P}_{n-k} -1 =p+\varepsilon - 1.
\end{equation*}
But for all $l \in \llbracket 1,k-1 \rrbracket $ we have that 
$\widehat{Y}_{n-l+1} \geq \widehat{Y}_{n-l}-1$, so that by an immediate induction, 
\begin{equation*}
\widehat{Y}_n\geq p+\varepsilon -1 -(k-1) \geq \varepsilon >0. 
\end{equation*}
\end{proof}
As a consequence of Proposition \ref{theolat}, determining when $\widehat{Y} = 0$ in Algorithm \ref{algo2} amounts to checking that the last $p$ arrivals are all $-1$'s, meaning that no item has entered the system in the last $p$ instants. 
In fact, as ii) above has a positive probability, we immediately see that in the present context, Algorithm \ref{algo2} terminates almost surely.

For any $x = x_1 \,\cdots\, x_p \in \widehat{\X}$, denote by $\overset{\circ}{x}$ the word $\overset{\circ}{x}_1\,\cdots\, \overset{\circ}{x}_p$, where for all $i \in \llbracket 1 , p \rrbracket,$ $\overset{\circ}{x}_i=x_i\mathbf 1_{x_i \ne -1}$. 
 The notions of synchronizing and strongly synchronizing words, are then extended as follows. 
\begin{definition}
A word $w\in (\mathbb{V}\cup {-1})^*$ is said to be {\em synchronizing} for the deterministic matching model with latency $(G,\Phi,\mu,p,\gamma)$ 
if \begin{equation}
\exists z \in \widehat{\X},\, \forall  x \in \widehat{\X},\,W^\Phi(x,w) = w_x \mbox{ is such that } \overset{\circ}{w_x} = z.
\end{equation}
We say that a word $w=w_1\cdots w_{2p}\in (\mathbb{V}\cup {-1})^*$ is {\em strongly synchronizing}, if 
\[\forall i \in \llbracket 1,p \rrbracket,\,\forall j \in \llbracket p+1, p+i \rrbracket,\, w_i \nleftrightline w_j,\]
{\em where, by convention,  $\forall v \in \V, v \nleftrightline -1$.}
\end{definition}

We can then apply the exact same arguments as for Theorems \ref{thm:sufficient} and \ref{thm:fcfm}, to show that 
\begin{proposition}
\label{theosync}
Any strongly synchronizing word $w \in (\V \cup -1)^*$ is also synchronizing word for the deterministic matching model with latency 
$(G,\Phi,\mu,p,\gamma)$. 
Conversely, if the matching policy $\Phi$ is {\sc fcfm}, then any synchronizing word of length $2p$  is strongly synchronizing. 
\end{proposition}

The previous result implies that Algorithm \ref{algo3} (by taking strongly synchronizing words in this new acception) terminates almost surely, and also produces a sample of the stationary distribution of the model with latency. 

As a conclusion, for a model with latency, both Algorithm \ref{algo2} and Algorithm \ref{algo3} are valid perfect sampling algorithms that terminate almost surely, and we can now compare their performance. For this, first notice that Algorithm \ref{algo3} is readily faster than the Algorithm \ref{algo2}, since the arrival of $p$ consecutive $-1$ also creates a strongly synchronizing word, as any word of length $2p$ containing $p$ `$-1$' as its first or last $p$ letters is strongly synchronizing. In Tables 6 to 9, we quantify this gain of applying Algorithm \ref{algo3} rather than Algorithm \ref{algo2} in terms of cpu time, for various parameters.
\begin{table}
\label{table:truc}
\begin{tabular}{|c|c|c|c|c|c|c|}
    \hline
    $p=3$& $q=\frac{1}{8}$& $q=\frac{2}{8}$& $q=\frac{3}{8}$& $q=\frac{4}{8}$& $q=\frac{5}{8}$ 
    \\
    \hline
    \hline
    $n = 4$, Algo \ref{algo2} & $0.0038045$&$0.0037805$ &$0.0040165$ &$0.0039802$ &$0.0042959$ 
    \\
    \hline
    $n=4$, Algo \ref{algo3} &$0.0004670 $&$0.0004578 $ &$0.0005623 $&$0.0007405 $ & $0.0009837 $ 
    \\
    \hline
    \hline
    $n=5$, Algo \ref{algo2}
    &$0.0044388$ & $0.0041361$ & $0.0043400$ & $0.0045547$ & $0.0043607$ 
        \\
    \hline
    $n=5$, Algo \ref{algo3}
      &$0.0004592$ &$0.0005210$ & $0.0005462$ & $0.0006469$&$0.0008396$ 
    \\
    \hline
    \hline
    $n=6$, Algo \ref{algo2}
      &$0.0046643$ &$0.0045620$ &$0.0046185$ &$0.0046222$ &$0.0048220$  
    \\
    \hline
    $n=6$, Algo \ref{algo3}
    &$0.0004305$ & $0.0004332$  & $0.0005398$ & $0.0006290$ & $0.0010127$ 
    \\
    \hline
    \hline
    $n=7$, Algo \ref{algo2}
    &$0.0049668$&$0.0053617$ &$0.0051719$ &$0.0048706$ &$0.0047918$ 
    \\
    \hline
    $n=7$, Algo \ref{algo3}
    & $0.0003674$ &$0.0005255$&$0.0006371$& $0.0008323$ & $0.0009967$ 
    \\
    \hline
    \hline
    $n=8$, Algo \ref{algo2}
    &$0.0049822$ &$0.0049859$ &$0.0050439$ &$0.0052366$ &$0.0059534$ 
    \\
    \hline
    $n=8$, Algo \ref{algo3}
    &$0.0003490$ &$0.0003939$ & $0.0006275$ &$0.0008972$ & $0.0013853$ 
    \\
    \hline
\end{tabular}
\captionof{table}{Average cpu time of the algorithms on standard computer for $100$ repetitions with $p=3, \gamma = 0.2$ and multiple values of $(n,q)$.}
\end{table}

\begin{table}
\begin{tabular}{|c|c|c|c|c|c|c|}
    \hline
    $p=3$& $q=\frac{1}{8}$& $q=\frac{2}{8}$& $q=\frac{3}{8}$& $q=\frac{4}{8}$& $q=\frac{5}{8}$ 
    \\
    \hline
    \hline
     $n=4$, Algo \ref{algo2} &$0.0003215 $&$0.0003227$ &$0.0003087$&$0.0003241 $ & $0.0002954 $ 
    \\
    \hline
    $n=4$, Algo \ref{algo3} &$0.0002715 $&$0.0002666$ &$0.0002210$&$0.0002102 $ & $0.0002593$ 
    \\
    \hline
    \hline
    $n=5$, Algo \ref{algo2}
    &$0.0003761$ & $0.0003799$ & $0.0003718$ & $0.0003690$ & $0.0003190$ 
        \\
    \hline
    $n=5$, Algo \ref{algo3}
      &0.0002000 & 0.0001950& 0.0001834& 0.0002102& 0.0002446 
    \\
    \hline
    \hline
 
    $n=6$, Algo \ref{algo2}
      &$0.0004253$ &$0.0003487$ &$0.0003704$ &$0.0003690$ &$0.0003768$  
    \\
    \hline
   
    $n=6$, Algo \ref{algo3}
    &$00.0002045$ & $0.0001953$  & $0.0002090$ & $0.0002409$ & $0.0002559$ 
    \\
    \hline
    \hline
   
    $n=7$, Algo \ref{algo2}
    &$0.0004097$&$0.0003916$ &$0.0004289$ &$0.0004132$ &$0.0004103$ 
    \\
    \hline
  
    $n=7$, Algo \ref{algo3}
    & $0.0002384$ &$0.0002094$&$0.0002073$& $0.0002868$ & $0.0002901$ 
    \\
    \hline
    \hline
    $n=8$, Algo \ref{algo2}
    &$0.0004186$ &$0.0004070$ &$0.0004908$ &$0.0005598$ &$0.0006104$ 
    \\
    \hline
    $n=8$, Algo \ref{algo3}
    &$0.0002074$ &$0.0001791$ & $0.0002156$ &$0.0002520$ & $0.0002659$ 
    \\
    \hline
\end{tabular}
\captionof{table}{Average cpu time of the algorithms on standard computer for $10000$ repetitions with $p=3,\gamma =0.5$ and multiple values of $(n,q)$.}
\end{table}

\begin{table}
\begin{tabular}{|c|c|c|c|c|c|c|}
    \hline
    $p=6$& $q=\frac{1}{8}$& $q=\frac{2}{8}$& $q=\frac{3}{8}$& $q=\frac{4}{8}$& $q=\frac{5}{8}$
    \\
    \hline
    \hline
    $n = 4$, Algo \ref{algo2} & $0.7051800$ & $0.5984500$  & $0.5163900$  & $0.4517600$  & $0.7003300 $
    \\
    \hline
    $n=4$, Algo \ref{algo3} &$0.0034300 $&$0.0057900 $ &$0.0056100 $&$0.0068200 $ & $0.0246900 $ 
    \\
    \hline
    \hline
    $n=5$, Algo \ref{algo2}
    &$0.5984500$ &$0.8970700$ &$0.5821500$ &$0.6060300$ &$0.5139100$ 
        \\
    \hline
    $n=5$, Algo \ref{algo3}
      &$0.0035900$ &$0.0026300$ &$0.0065800$ &$0.0109600$ &$0.0221800$ 
    \\
    \hline
    \hline
    $n=6$, Algo \ref{algo2}
      &$0.7514100$ &$0.6232700$ &$0.6096600$ &$0.9526300$ &$0.7364400$ 
    \\
    \hline
    $n=6$, Algo \ref{algo3}
    &$0.0009300$ & $0.0042500$ &$0.0051800$ &$0.0105000$ & $0.0219500$ 
    \\
    \hline
    \hline
    $n=7$, Algo \ref{algo2}
    &$0.7517200$ &$0.5460800$ &$0.4714300$ &$0.5977900 $&$0.6670700$ 
    \\
    \hline
    $n=7$, Algo \ref{algo3}
    &$0.0014000$ &$0.0025000$ & $0.0042300$ & $0.0084900$ &$0.0171500$  
    \\
    \hline
    \hline
    $n=8$, Algo \ref{algo2}
    &$0.5546900$ &$0.6359300$ &$0.4876600$ &$0.7620300$ &$0.6122400$ 
    \\
    \hline
    $n=8$, Algo \ref{algo3}
    &$0.00109000$ &$0.0017500$ &$0.0037400$ & $0.0092300$ &$0.0376300$ 
    \\
    \hline
\end{tabular}
\captionof{table}{Average cputime of the algorithms for $100$ repetitions with $p=6,\gamma = 0.2$ and multiple values of $(n,q)$.}
\end{table}

\begin{table}
\begin{tabular}{|c|c|c|c|c|c|c|}
    \hline
    $p=6$& $q=\frac{1}{8}$& $q=\frac{2}{8}$& $q=\frac{3}{8}$& $q=\frac{4}{8}$& $q=\frac{5}{8}$
    \\
    \hline
    \hline
    $n = 4$, Algo \ref{algo2} & $0.0046900 $&$0.00316360 $ &$0.0034701 $ &$0.0031292 $ &$0.0031772 $ 
    \\
    \hline
    $n=4$, Algo \ref{algo3} &$0.0010425 $&$0.0010735 $ &$0.0012352 $&$0.0013639 $ & $0.0016111 $ 
    \\
    \hline
    \hline
    $n=5$, Algo \ref{algo2}
    &$0.0040600$ &$0.0035472$ &$0.0035068$ &$,0.0034339$ &$0.0035351$ 
        \\
    \hline
    $n=5$, Algo \ref{algo3}
      &$0.0008387$ &$0.0010067$ &$0.0010954$ &$0.0013510$ &$0.0015984$ 
    \\
    \hline
    \hline
   
    $n=6$, Algo \ref{algo2}
      &$0.0034200$ &$0.0035355$ &$0.0039241$ &$0.00358213$ &$0.0036049$ 
    \\
    \hline
    $n=6$, Algo \ref{algo3}
    &$0.0008432$ &$0.0008729$ &$0.0010616$ &$0.0013186$ & $0.0017617$ 
    \\
    \hline
    \hline
    $n=7$, Algo \ref{algo2}
    &$0.0048400$ &$0.0037248$ &$0.0037925$ &$0.0036611 $&$0.0038020$ 
    \\
    \hline
    $n=7$, Algo \ref{algo3}
    &$0.0007764$ &$0.0008591$ & $0.0010278$ & $0.0013062$ &$0.0017764$  
    \\
    \hline
    \hline
    $n=8$, Algo \ref{algo2}
    &$0.0054600$ &$0.0038513$ &$0.0038263$ &$0.0037788$ &$0.0038358$ 
    \\
    \hline
    $n=8$, Algo \ref{algo3}
    &$0.0007209$ &$0.0008657$ &$0.0010599$ & $0.0013521$ &$0.0018399$ 
    \\
    \hline
\end{tabular}
\captionof{table}{Average cputime of the algorithms on standard computer for $10000$ repetitions with $p=6,\gamma = 0.5$ and multiple values of $(n,q)$.}
\end{table}

\subsection{Estimating of the loss probability for {\sc ml} and {\sc fcfm}}
\label{subsubsec:compare}

Algorithm \ref{algo3} returns a random variable that is distributed from the stationary distribution of the system. This result can be of critical use, to compare the performance of systems, for which no exact characterization of the steady state is known.  As an example, we are able to assess the asymptotic loss rate of items of every class. We use this to compare two matching policies in steady state: Match the Longest ({\sc ml}) and First Come, First Matched ({\sc fcfm}). 

Let $(G = (\mathbb{V},E),\Phi,\mu,p)$ be a discrete matching model with deterministic impatience, and $\tilde X= \suitez{\tilde X_n}$ be the Markov chain of the system. Let $\pi$ be the stationnary distribution for $X$ and for all $(i,j) \in \mathbb{V}^2$ such that $(i,j)\not\in E$, 
$$A_{i,j} = \lbrace x=x_1\cdots x_j \in \X, \,x_1 = i \text{ and the arrival is of class } j\mbox{ in a buffer }x\rbrace.$$ 
The asymptotic loss rate of items of class $i$ is denoted by 
\begin{equation}
\rho(i) :=\lim\limits_{N \rightarrow +\infty} \frac{\sum\limits_{n=1}^N \mathds{1}_{A^i_n}}{N},
\end{equation}
where for all $n$, 
 $$A^i_n = \lbrace \text{An item of class } i \text{ is lost at time } n  \rbrace.$$
 An immediate first step analysis implies that 
\begin{equation}
\label{eqloss}
\rho(i) = \sum \limits_{j \in \mathbb{V}}  \pi(A_{i,j}) \mu(j),
\end{equation}
so $\rho(i)$ can also be interpreted as the probability to lose an item of class $i$ in the system at a given instant, in steady state. 
Reasoning similarly, $$\rho = \sum\limits_{i \in \mathbb{V}} \rho(i)$$ is the asymptotic loss rate of items (of any class) in the system, and can also be seen as the probability to lose an item (of any class) at a given time, in steady state. 
Using equation \ref{eqloss} we can then estimate those asymptotic loss rate by  running our perfect simulation algorithm \ref{algo3}, and then estimating $\pi(A_{i,j})$ for all  $i,j \in \mathbb{V}$, by a Monte-Carlo estimate. 

Table 4 presents the results over $10^4$ simulations, for $G$ a random Erdös-Renyi graph of parameters $n = 5, q = 0.6$, conditioned on being connected, for $p = 5$, and for $\mu$ the uniform distribution. Both matching policies {\sc fcfm} and {\sc ml} are implemented on the same 
samples each time. We observe that the overall asymptotic loss rate is slightly, but consistently lower under {\sc fcfm} than under {\sc ml}, although nominal loss rates of given nodes can be higher under {\sc fcfm}.


\begin{table}
\begin{tabular}{|c|c|c|c|c|c|c|}
     \hline
      & $\rho$ &$\rho(1)$ & $\rho(2)$  & $\rho(3)$ & $\rho(4)$& $\rho(5)$ \\
     \hline
      {\sc fcfm}  &0.0293&0.00026& 0.00032  & 0.0132  & 0.0152 & 0.00032  \\
     \hline
     {\sc ml} & 0.03122 &0.00028 & 0.0003  & 0.01536  & 0.01486 & 0.00042   \\
     \hline
\end{tabular}
\captionof{table}{MC Estimates for the asymptotic loss rates for $10^4$ repetitions of Algorithm \ref{algo3} for a random Erdös-Renyi graph of parameters 
$n = 5, q = 0.6 $, for $p = 5$ and $\mu$ the uniform distribution.}
\end{table}

\bibliographystyle{acm} 
\bibliography{bibliomatchingsimuparfaitethomasmasanet}
\end{document}